\documentclass[a4paper,11pt,reqno]{amsart}

\usepackage[headings]{fullpage}

\usepackage{dsfont,amssymb}
\usepackage[all]{xy}
\usepackage{mathdots}
\usepackage{xcolor}

\usepackage{hyperref}

\usepackage[backend    = biber,
                maxnames   = 5,
				sortcites  = true,
                giveninits = true,
                doi        = false,
                isbn       = false,
                url        = false,
                style      = alphabetic]{biblatex}
\DeclareNameAlias{sortname}{family-given}
\DeclareNameAlias{default}{family-given}
\renewbibmacro{in:}{%
  \ifentrytype{article}{}{\printtext{\bibstring{in}\intitlepunct}}}

\DeclareFontFamily{U}{matha}{\hyphenchar\font45}
\DeclareFontShape{U}{matha}{m}{n}{
      <5> <6> <7> <8> <9> <10> gen * matha
      <10.95> matha10 <12> <14.4> <17.28> <20.74> <24.88> matha12
      }{}
\DeclareSymbolFont{matha}{U}{matha}{m}{n}
\DeclareFontFamily{U}{mathb}{\hyphenchar\font45}
\DeclareFontShape{U}{mathb}{m}{n}{
      <5> <6> <7> <8> <9> <10> gen * mathb
      <10.95> mathb10 <12> <14.4> <17.28> <20.74> <24.88> mathb12
      }{}
\DeclareSymbolFont{mathb}{U}{mathb}{m}{n}
\DeclareMathSymbol{\ovoid}{3}{matha}{"6C}
\DeclareMathSymbol{\boxvoid}{2}{mathb}{"6C}

\mathchardef\mhyph="2D

\numberwithin{equation}{section}

\newtheorem{theorem}{Theorem}[section]
\newtheorem{corollary}[theorem]{Corollary}
\newtheorem{lemma}[theorem]{Lemma}
\newtheorem{proposition}[theorem]{Proposition}
\theoremstyle{remark}
\newtheorem{remark}[theorem]{Remark}
\newtheorem{example}[theorem]{Example}
\theoremstyle{definition}
\newtheorem{definition}[theorem]{Definition}

\newcommand\bp{\begin{proof}}
\newcommand\ep{\end{proof}}
\newcommand\ee{\nopagebreak\mbox{\ }\hfill$\diamond$}

\newcommand\diag{\operatorname{diag}}
\DeclareMathOperator{\Mat}{Mat}
\DeclareMathOperator{\Rep}{Rep}
\DeclareMathOperator{\Tr}{Tr}

\newcommand{\C}{{\mathbb C}}
\newcommand{\Z}{{\mathbb Z}}
\newcommand{\R}{{\mathbb R}}
\newcommand\T{{\mathbb T}}
\newcommand\eps{\varepsilon}

\newcommand{\hopf}{\mathrm{Hopf}^*}

\begin{document}

\title{Cocycle twisting of semidirect products and transmutation}

\author{Erik Habbestad}
\address{Department of Mathematics, University of Oslo, P.O. Box 1053
Blindern, NO-0316 Oslo, Norway}

\email{erikhab@math.uio.no}

\author{Sergey Neshveyev}
%\address{Universitetet i Oslo}
\email{sergeyn@math.uio.no}

\thanks{Supported by the NFR project 300837 ``Quantum Symmetry''.}

\begin{abstract}
We apply Majid's transmutation procedure to Hopf algebra maps $H \to \C[T]$, where $T$ is a compact abelian group, and explain how this construction gives rise to braided Hopf algebras over quotients of $T$ by subgroups that are cocentral in $H$. This allows us to unify and generalize a number of recent constructions of braided compact quantum groups, starting from the braided $SU_q(2)$ quantum group, and describe their bosonizations.
\end{abstract}

\date{April 2, 2023; revised: January 11, 2024}

\maketitle

\section*{Introduction}

Assume that $H$ and $K$ are Hopf algebras together with Hopf algebra maps
\[ p\colon H \to K, \quad i\colon K \to H, \quad p \circ i = \mathrm{id}_K.\]
This setting was considered in 1985 by Radford in his paper \textit{The structure of Hopf algebras with a projection}, \cite{Radford}.
He explained that this data is equivalent to having an object $A$ in the category of Yetter--Drinfeld modules over $K$ satisfying certain conditions reminiscent of a Hopf algebra. This object is often not a genuine Hopf algebra, because the algebra structure considered on $A \otimes A$ involves the braiding on the Yetter--Drinfeld modules. Today, following Majid, it is common to call $A$ a braided Hopf algebra. It is an example of a Hopf algebra object in a braided monoidal category.

As for genuine Hopf algebras, there is a Tannaka--Krein type reconstruction theorem for braided Hopf algebras \cite[Chapter 9]{Majid}. That is, under certain representability conditions, a monoidal functor $\mathcal{C} \to \mathcal{D}$, where $\mathcal{C}$ is rigid and $\mathcal{D}$ is braided, gives rise to a Hopf algebra object in $\mathcal{D}$. A particular example of this is when the functor is induced from a map $\pi\colon H \to K$, where~$H$ is a Hopf algebra and $K$ is a coquasitriangular Hopf algebra: the natural monoidal functor between the categories of comodules \[\mathcal{F}_\pi\colon \mathcal{M}^H \to \mathcal{M}^K,\]
gives rise to a braided Hopf algebra, which Majid calls a \textit{transmutation} of $H$, \cite{Majid-bg}.

More recently, there have been a number of constructions of \textit{braided compact quantum groups}, \cite{ABRR}, \cite{BJR}, \cite{KMRW}, \cite{MR-ort}. For the authors of the present paper they showed up ``in nature'' through a connection with C$^*$-algebras arising from certain subproduct systems, \cite{HN22}. In particular, a Cuntz--Pimsner type algebra associated to the noncommutative polynomial $X_1X_2-\bar q X_2X_1$ turned out to be the C$^*$-algebra of continuous functions on the braided $SU_q(2)$ quantum group, constructed in \cite{KMRW}. Although the analytic/C$^*$-algebraic aspects are therefore important to us, compactness nevertheless allows one to treat an essential part of the theory purely algebraically.

In this paper our starting point is a map $\pi\colon H \to \C[T]$, where $H$ is a Hopf $*$-algebra and $T$ is a compact abelian group. We observe that the resulting transmutation may be viewed as a braided Hopf $*$-algebra over the quotients of $T$ by subgroups $T_0$ that are cocentral in $H$. Using a theorem due to Majid we describe the corresponding Hopf algebra with projection, called \textit{bosonization}, in terms of $2$-cocycle twists of $\C[T/T_0]\ltimes H$. This allows us to treat a number of examples in a unified and efficient way.

The paper is organized as follows. In Section \ref{s1} we collect some facts about braided Hopf $*$-algebras, twisting and transmutation. Section \ref{s2} contains our general results. Here we also suggest a definition of \emph{braided compact matrix quantum groups}, covering in particular transmutations of compact matrix quantum groups, and discuss a connection of transmutation with a recent construction of Bochniak and Sitarz \cite{BS}. Section \ref{s3}, which constitutes a large part of this text, consists of computing examples.
%In particular, we explain how we can recover and generalize many of the braided compact quantum groups %already considered in the literature.

\bigskip

\section{Generalities}\label{s1}
We refer the reader to Majid's book \cite{Majid} for more information on the notions in this section. We only consider unital algebras over the complex numbers.

\medskip

\subsection{Categories of comodules} Let $(H, \Delta, S, \varepsilon)$ be a Hopf $*$-algebra. As is usual in the theory of Hopf algebras, we adopt Sweedler's sumless notation: we write $\Delta(h) = h_{(1)} \otimes h_{(2)}$, but remember that the expression represents a sum of simple tensors. We write $\mathcal{M}^H$ for the monoidal category of right $H$-comodules. For an object $M \in \mathcal{M}^H$ with corresponding map $\delta_M\colon M \to M \otimes H$, we write $\delta_M(m) = m^{(1)} \otimes m^{(2)}$.

Recall that a coquasitriangular structure on $H$ is a linear map $R\colon H \otimes H \to \C$ which is convolution invertible, with convolution inverse $R^{-1}$, and satisfies
\[R(xy,z) = R(x,z_{(1)})R(y,z_{(2)}), \quad R(x,yz) = R(x_{(1)}, z)R(x_{(2)},y), \]
\[y_{(1)}x_{(1)}R(x_{(2)},y_{(2)}) = R(x_{(1)},y_{(1)})x_{(2)}y_{(2)}\]
for all $x,y,z \in H$. We say that $R$ is unitary if it is unitary as an element of the convolution $*$-algebra $(H\otimes H)^*$, or equivalently, as $R=R\circ(S\otimes S)$,
$$
R^{-1}(x,y) = \overline{R(x^*,y^*)}.
$$

For the rest of this subsection we assume that $H$ is given a unitary coquasitriangular structure~$R$. Then $\mathcal{M}^H$ has a braiding given by \[M \otimes N \to N \otimes M, \quad m \otimes n \mapsto R(m^{(2)}, n^{(2)}) n^{(1)} \otimes m^{(1)}.\]
There is also an induced right $H$-module structure on $M$ given by
\begin{equation}\label{eq:right-action}
m\triangleleft h = R(m^{(2)},h) m^{(1)}, \quad h \in H,\ m \in M.
\end{equation}
The right comodule and module structures are compatible in the sense that
$$
\delta_M(m\triangleleft h)=m^{(1)}\triangleleft h_{(2)}\otimes S(h_{(1)})m^{(2)}h_{(3)},
$$
so $M$ becomes a Yetter--Drinfeld module over $H$.

We denote by $\mathrm{Alg}^*(H)$ the category of right $H$-comodule $*$-algebras. In other words, an object in $\mathrm{Alg}^*(H)$ is a $*$-algebra $A$ and a right $H$-comodule such that the map $\delta_A\colon A \to A\otimes H$ is a $*$-homomorphism. For $A, B \in \mathrm{Alg}^*(H)$ we denote by $A \otimes_R B$ the $*$-algebra with underlying vector space $A \otimes B$ equipped with the product
\[(a \otimes b)\cdot(a' \otimes b') = R(b^{(2)}, a'^{(2)})aa'^{(1)} \otimes b^{(1)}b' \]
and the $*$-structure
\begin{equation}\label{eq:braided-star}
(a\otimes b)^* = (1\otimes b^*)\cdot(a^* \otimes 1)=R(b^{(2)*},a^{(2)*})a^{(1)*}\otimes b^{(1)*}.
\end{equation}
The braided tensor product $\otimes_R$ turns $\mathrm{Alg}^*(H)$ into a monoidal category $\mathrm{Alg}^*(H,R)$ with equivariant $*$-homomorphisms as morphisms.

Considering $H$ as an $H$-comodule $*$-algebra with the coaction given by the coproduct, we recover the \textit{smash product} (with respect to the right action~\eqref{eq:right-action}):
\[H \# A = H \otimes_R A.\]

Let $\hopf(H,R)$ denote the category of Hopf $*$-algebras internal to the braided monoidal category $(\mathcal{M}^H,R)$. An object $A \in \hopf(A,R)$ is thus an $H$-comodule $*$-algebra together with $H$-comodule maps
\[\Delta_A\colon A \to A \otimes_R A, \quad  S_A\colon A \to A, \quad \varepsilon_A\colon A \to \C, \]
which are required to fit in commutative diagrams analogous to those defining Hopf $*$-algebras. An object in $\hopf(H,R)$ is called a \textit{braided Hopf $*$-algebra}, and it is usually not a genuine Hopf $*$-algebra. It is, however, always closely related to one:

\begin{definition}
\label{def:bos}
The \textit{bosonization} of $A \in \hopf(H,R)$ is the Hopf $*$-algebra with underlying $*$-algebra $H \#  A$, counit $\varepsilon(h\# a) = \varepsilon_H(h)\varepsilon_A(a)$, coproduct
\begin{equation}
\label{eq:cosmash}
\Delta(h \# a) = (h_{(1)} \# {a_{(1)}}^{(1)}) \otimes (h_{(2)}{a_{(1)}}^{(2)}\# a_{(2)}),
\end{equation}
and antipode
\[S(h\#a) = (1 \# S_A(a^{(1)}))(S_H(ha^{(2)}) \# 1).\]
\end{definition}

That $H \# A$ is a Hopf algebra, is a special case of results of Radford~\cite{Radford}, who proved that a Hopf algebra object in the category of Yetter--Drinfeld-modules is equivalent to a \textit{Hopf algebra with projection}.

\medskip

Let $A$ be a braided Hopf $*$-algebra. We will only consider $A$-comodules internal to $\mathcal{M}^H$. Thus, the notion of an \textit{$A$-comodule} will be reserved for triples $(M,\delta_M, \gamma_M)$, where $\delta_M\colon M \to M \otimes H$ is an $H$-comodule and $\gamma_M\colon M \to M \otimes A$ is a morphism of $H$-comodules that defines a comodule for the coalgebra $A$ in the usual sense.
We record the following well-known result:

\begin{proposition}[cf.~{\cite[Theorem~9.4.12]{Majid}}]
\label{prop:braided-comodule}
Let $A \in \hopf(H,R)$. Then the category of $A$-comodules is isomorphic to the category of $(H \# A)$-comodules through the assignment
\[ (M,\delta_M, \gamma_M) \mapsto (M,(\delta_M \otimes \iota) \gamma_M).\]
The inverse is given by
\[ (M,\delta) \mapsto (M,  (\iota \otimes (\iota\# \varepsilon_A))\delta,(\iota\otimes (\varepsilon_H\# \iota))\delta). \]
\end{proposition}

We will say that a finite dimensional $A$-comodule is \emph{unitary}, if the corresponding $(H\#A)$-comodule is unitary. Recall that a finite dimensional comodule $(M,\delta)$ over a Hopf $*$-algebra $H'$ is called unitary, if $M$ is equipped with a scalar product and
$$
\langle\delta(m),\delta(m)\rangle=(m,m')1\ \ \text{for all}\ \ m,m'\in M,
$$
where the $H'$-valued sesquilinear form $\langle\cdot,\cdot\rangle$ is defined by $\langle m\otimes a,m'\otimes b\rangle=(m,m')b^*a\in H'$.

\subsection{Twisting and transmutation} Assume that $J\colon H \otimes H \to \C$ is a Hopf 2-cocycle. This means that $J$ is convolution invertible and satisfies
\[ J(x_{(1)},y_{(1)})J(x_{(2)}y_{(2)},z) = J(y_{(1)}, z_{(1)})J(x, y_{(2)}z_{(2)})\]
for $x, y, z \in H$. We say that $J$ is unitary if $J^* = J^{-1}$ in the convolution $*$-algebra $(H \otimes H)^*$, that is,
$$
J^{-1}(x,y)=\overline{J(S(x)^*,S(y)^*)}.
$$
Given $J$, we can define a convolution invertible element $u\colon H \to \C$ by
\[ u(x) = J(x_{(1)},S(x_{(2)})).\]
The following is the well-known twisting procedure for Hopf $*$-algebras, see, e.g.,~\cite[Theorem~2.3.4]{Majid}.
\begin{proposition}
\label{prop:cocycle-twist}
Let $J$ be a unitary Hopf 2-cocycle on $H$. There is a Hopf $*$-algebra ${}_JH_{J^{-1}}$ having the same coalgebra structure as $H$ and new product, antipode and involution defined by
\[ x \star_J y  = J(x_{(1)}, y_{(1)})x_{(2)}y_{(2)}J^{-1}(x_{(3)}, y_{(3)}),\]
\[ S^J(x) = u(x_{(1)})S(x_{(2)})u^{-1}(x_{(3)}), \quad x^{*_J} = u^{-1}(S(x_{(1)})^*)x^*_{(2)}u(S(x_{(3)})^*)=S^J(S^{-1}(x^*)).\]
\end{proposition}

We remark that we are interested in unitary cocycles, while Majid in~\cite[Section~2.3]{Majid} considers so-called real ones. That the $*$-structure on ${}_JH_{J^{-1}}$ defined above is the correct one in the unitary case is explained in~\cite[Example~~2.3.9]{NT13} in the context of compact quantum groups. It is not difficult to check directly that the definition works for general Hopf $*$-algebras.

\smallskip

Next we consider a way to produce braided Hopf algebras, due to Majid \cite{Majid-bg}. Recall that the adjoint comodule for $H$ is given by
\[\mathrm{ad}\colon H \to H \otimes H, \quad \mathrm{ad}(x) = x^{(1)} \otimes x^{(2)} = x_{(2)} \otimes  S(x_{(1)})x_{(3)}.\]

Let us for the moment forget about the $*$-structure on $H$. The process of \textit{transmutation} produces a braided Hopf algebra from the $H$-comodule $(H,\mathrm{ad})$.

\begin{proposition}[{\cite[Theorem~4.1]{Majid-bg}}]
\label{prop:transmutation}
Assume $H$ is a Hopf algebra with a coquasitriangular structure $R\colon H \otimes H \to \C$. Then $H$ with the same coalgebra structure and new product $\cdot_R$ and antipode $S_R$, given by
\[x \cdot_R y = x_{(2)}y_{(3)}R(x_{(3)}, S(y_{(1)}))R(x_{(1)}, y_{(2)}), \]
\[S_R(x) = S(x_{(2)})R(S^2(x_{(3)})S(x_{(1)}), x_{(4)}), \]
defines a braided Hopf algebra $H_R$ over $(H,R)$, where the comodule structure is given by $\mathrm{ad}$.
\end{proposition}

Note that even though the coproduct is not changed, it is now considered as an algebra map $H \to H \otimes_R H$.

\smallskip

More generally, assume that $\pi\colon H \to K$ is a map of Hopf algebras, where $K$ has a coquasitriangular structure $R$. Although $\pi^*R=R\circ(\pi\otimes\pi)$ is not a coquasitriangular structure on $H$ in general, the same formulas as above, with $R$ replaced by $\pi^*R$, define a braided Hopf algebra~$H_R$ over $(K,R)$ with the restricted coaction
\begin{equation}
\label{eq:adjoint-restriction}
\mathrm{ad}_\pi = (\iota\otimes\pi)\mathrm{ad}\colon H \to H \otimes K.
\end{equation}

The next result relates the bosonization of $H_R$ to a Hopf algebra with the tensor product coalgebra structure.

\begin{theorem}
\label{thm:cocycle-transmutation}
Assume that $\pi\colon H \to K$ is a map of Hopf algebras and $R \colon K \otimes K \to \C$ is a coquasitriangular structure on $K$. Let $K \bowtie_R H$ denote the Hopf algebra that coincides with $K\otimes H$ as a coalgebra, but has the twisted product
\[(k \otimes h) \cdot (k' \otimes h') = R^{-1}(\pi(h_{(1)}),k_{(1)}') \, kk_{(2)}' \otimes h_{(2)}h' \, R(\pi( h_{(3)}), k_{(3)}').\]
Then the map
\[K \bowtie_R H\to K \# H_R, \quad k \otimes h \mapsto k\pi(h_{(1)}) \# h_{(2)}, \]
defines an isomorphism of Hopf algebras.
\end{theorem}
\begin{proof}
For $K=H$ and $\pi=\operatorname{id}$, this is~\cite[Theorem 7.4.10]{Majid}. The general case can be proved similarly or by a direct computation.
\end{proof}

Note that $K \otimes \C 1$ is a Hopf subalgebra of $K \bowtie_R H$. Therefore $K \bowtie_R H$ is a Hopf algebra with projection $K \bowtie_R H\to K$, $k\otimes h\mapsto k\pi(h)$, and then the transmuted braided Hopf algebra~$H_R$ can be viewed as a particular case of the construction of Radford~\cite{Radford}.

\begin{remark}\label{rem:cocycle}
The formula for the product on $K\bowtie_R H$ is the same as for the cocycle twisting~by
$$
J((k \otimes h) , (k' \otimes h'))=\eps_K(k)\eps_H(h')R^{-1}(\pi(h), k'),
$$
but the element $J$ is not a Hopf $2$-cocycle on $K\otimes H$ in general. It becomes a $2$-cocycle, when~$K$ is cocommutative. Note also that if $K$ is both commutative and cocommutative, then $H$ is a $K$-comodule algebra with respect to the adjoint coaction and then as an algebra $K\bowtie_R H$ coincides with $K\# H=K\otimes_R H$.
\end{remark}

\bigskip

\section{Transmutation over abelian groups}\label{s2}
Let $(H, \Delta, S, \varepsilon)$ be a Hopf $*$-algebra and $T$ be a compact abelian group. We write $\hat{T}$ for the dual discrete group and $\C\hat{T}$ for the group Hopf $*$-algebra of $\hat{T}$. We will frequently identity $\C\hat{T}$ with the function algebra $\C[T]$, and will use the latter notation when it is natural to focus on the compact group $T$. Throughout this section we assume that we are given a Hopf $*$-algebra map $\pi\colon H \to \C\hat{T}$.

\subsection{Braided Hopf algebras over quotients of \texorpdfstring{$T$}{T}}
There are canonical commuting left and right coactions
\[ \delta_L = (\pi \otimes \iota)\Delta, \quad \delta_R = (\iota\otimes \pi)\Delta \] by $\C\hat{T}$ on $H$. It follows that $H$ is bi-graded by $\hat{T}$. More precisely, $H= \bigoplus_{a,b \in \hat{T}} H_{a,b}$, where
\[ H_{a,b} = \{ x \in H \, | \, \delta_L(x) = a\otimes x \mbox{ and } \delta_R(x) = x \otimes b \}. \]
A consequence of coassociativity is then that for any $x \in H$ we have
\begin{equation}
\label{eq:coproduct-homo}
\Delta(x) = x_{(1)}\otimes x_{(2)} \in \bigoplus_{a,b,c} H_{a,b} \otimes H_{b,c}.
\end{equation}
We use the shorthand notation $x_{(1)}^{a,b} \otimes x_{(2)}^{b,c}$ to mean the part of the sum $x_{(1)} \otimes x_{(2)}$ which is in $H_{a,b} \otimes H_{b,c}$. Thus, $\Delta(x) = \sum_{a,b,c} x_{(1)}^{a,b} \otimes x_{(2)}^{b,c}$.

\medskip

Consider a closed subgroup $T_0 \subset T$, with the corresponding restriction map $q\colon \C[T] \to \C[T_0]$. We say that $T_0$ is $H$-\textit{cocentral}, or that the map $q\pi$ is cocentral, if the induced adjoint coaction
\[ \mathrm{ad}_{q\pi} = (\iota\otimes q\pi)\mathrm{ad}\colon H \to H \otimes \C[T_0] \]
is trivial. In other words,
$$
H_{a,b}=0\quad\text{whenever}\quad q(a)\ne q(b).
$$
This condition implies that if we view $\C[T/T_0]$ as a Hopf $*$-subalgebra of $\C[T]$, then
\begin{equation}
\label{diag:ad}
\mathrm{ad}_\pi(H)\subset H\otimes\C[T/T_0],
\end{equation}
so that $(H,\mathrm{ad}_\pi)$ can be viewed as a $\C[T/T_0]$-comodule.

Denote by $\Delta(T_0)$ the diagonal subgroup in $T_0 \times T_0$. It is $(\C[T]\otimes H)$-cocentral with respect to the composition of $\iota\otimes\pi\colon\C[T]\otimes H\to \C[T]\otimes\C[T]$ with the restriction map $\C[T\times T]\to \C[\Delta(T_0)]$. It follows that the spaces of right and left coinvariants coincide,
\begin{equation}\label{eq:coinvariants}
(\C[T] \otimes H)^{\Delta(T_0)}={}^{\Delta(T_0)}(\C[T] \otimes H)=\bigoplus_{\substack{a,b,c:\\ q(ab)=1}}a\otimes H_{b,c}
=\bigoplus_{\substack{a,b,c:\\ q(ac)=1}}a\otimes H_{b,c},
\end{equation}
and define a Hopf $*$-subalgebra of $\C[T]\otimes H$.

In the present setting we observe that $H$ equipped with $\mathrm{ad}_\pi$ is an object in $\hopf(\C[T], \eps\otimes\eps)$ without any modifications.
%it is the ``trivial transmutation'' (recall Proposition~\ref{prop:transmutation}).
Moreover, by~\eqref{diag:ad} it can also be viewed as an object in $\hopf(\C[T/T_0],\eps\otimes\eps)$.  The corresponding bosonization $\C[T/T_0] \ltimes H$ is just the tensor product $*$-algebra with the coproduct defined by~\eqref{eq:cosmash} (in other words, as a coalgebra, it is the \emph{smash coproduct} of $\C[T/T_0]$ and $H$):
$$
\Delta_{\C[T/T_0] \ltimes H}(a\otimes x)=\sum_d(a\otimes x^{b,d}_{(1)})\otimes (ab^{-1}d\otimes x^{d,c}_{(2)})
$$
for $a\in\widehat{T/T_0}\subset\hat T$ and $x\in H_{b,c}$.

\begin{proposition}
\label{prop:main-untwisted}
Consider $H$ as a $\C[T]$-comodule Hopf $*$-algebra under the adjoint coaction~$\mathrm{ad}_\pi$.  Then
\[\Theta\colon \C[T] \otimes H \to \C[T] \ltimes H, \quad \Theta(a \otimes x) = a\pi(x_{(1)})\otimes x_{(2)},\]
is an isomorphism of Hopf $*$-algebras. Moreover, $\Theta$ restricts to an isomorphism of Hopf $*$-algebras
\[ (\C[T] \otimes H)^{\Delta(T_0)} \cong \C[T/T_0] \ltimes H\]
for any $H$-cocentral closed subgroup $T_0 \subset T$.
\end{proposition}
\begin{proof}
That $\Theta$ is an isomorphism is easily verified, but except for the $*$-structure it is also a special case of Theorem~\ref{thm:cocycle-transmutation}.  As $\Theta(a \otimes H_{b,c}) = ab \otimes H_{b,c}$ and $\C[T/T_0]$ is spanned by $a\in\hat T$ such that $q(a)=1$, the second part of the proposition follows from~\eqref{eq:coinvariants}.
\end{proof}

\begin{remark}
Let $G$ be a compact group with a closed abelian subgroup $T$. Assume that $T_0$ is a closed subgroup of $T \cap Z(G)$, where $Z(G)$ is the center of~$G$. As $T_0$ acts trivially under the conjugation action $(t,g) \mapsto tgt^{-1}$ of $T$ on $G$, we have an induced action on $G$ by the quotient group~$T/T_0$. Then the map
\[(T/T_0) \ltimes G\to (T \times G)/\Delta(T_0), \quad ([t], g)\mapsto [(t,tg)],\]
is a group isomorphism. The above result can be seen as a generalization of this.\ee
\end{remark}

Next, fix a bicharacter $\beta\colon \hat{T} \times \hat{T} \to \T$.
This simply means that, for all $a,b,c$ in $\hat{T}$,
\[ \beta(ab,c) = \beta(a,c)\beta(b,c) \quad \mbox{ and } \quad \beta(a,bc) = \beta(a,b)\beta(a,c). \]
We will write $\beta$ also for the extension of the bicharacter to a linear map $\C\hat{T} \otimes \C\hat{T} \to \C$. This defines a unitary coquasitriangular structure on $\C[T] = \C\hat{T}$. We want to understand the corresponding transmutation $H_\beta\in\hopf(\C[T],\beta)$.

By definition, the product on $H_\beta$ is determined by
\begin{equation}
\label{eq:trans-prod}
x \cdot_\beta y = \beta(a^{-1}b,c^{-1})xy, \quad x \in H_{a,b},\ y \in H_{c,d},
\end{equation}
the coproduct and counit remain unchanged, while the antipode is determined by
\begin{equation}
\label{eq:trans-antipode}
S_\beta(x) = \beta(a^{-1}b,b)S(x), \quad x \in H_{a,b}.
\end{equation}
In the present setting we may also introduce a $*$-structure on $H_\beta$:
\begin{lemma}
The formula
\begin{equation}\label{eq:trans-star}
x^{*_\beta} = \beta(a^{-1}b,a^{-1})x^*, \quad x \in H_{a,b},
\end{equation}
turns $H_\beta$ into a braided Hopf $*$-algebra.
\end{lemma}
\begin{proof}
Using that $x^*\in H_{a^{-1},b^{-1}}$ for $x\in H_{a,b}$, it is straightforward to check that $*_\beta$ is involutive. Using formula~\eqref{eq:trans-prod}, it is also easy to see that $(x \cdot_\beta y)^{*_\beta} = y^{*_\beta} \cdot_\beta x^{*_\beta}$. To show that $\Delta\colon H_\beta \to H_\beta \otimes_\beta H_\beta$ is $*$-preserving, notice that if $x \in H_{a,b}$ and $y \in H_{b,c}$, then by~\eqref{eq:braided-star} we have
\begin{align*}
(x \otimes y)^{*_\beta \, \otimes_\beta \, *_\beta}
&= \beta(a^{-1}b,a^{-1})\beta(b^{-1}c,b^{-1})\beta(b^{-1}c,a^{-1}b)\,x^* \otimes y^* \\
&= \beta(a^{-1}c,a^{-1})\, x^* \otimes y^*.
\end{align*}
That $\Delta$ preserves the new $*$-structure is thus a consequence of~\eqref{eq:coproduct-homo}.
\end{proof}

By~\eqref{diag:ad} we thus get the following:
\begin{proposition}
\label{prop:braided-over-quotient}
For any $H$-cocentral closed subgroup $T_0 \subset T$, the transmutation $H_\beta$ can be viewed as an object in $\hopf(\C[T/T_0], i^*\beta)$, where $i\colon\C[T/T_0]\to\C[T]$ is the embedding map.
\end{proposition}

The following is the main result of this section.
\begin{theorem}
\label{thm:main}
Given a bicharacter $\beta\colon\hat T\times\hat T\to\T$, let $J_1, J_2\colon (\hat{T}\times \hat{T}) \times (\hat{T}\times \hat{T}) \to \T$ be the $2$-cocycles defined by
\[J_1((a,b),(c,d)) = \overline{\beta(b,c)}, \quad J_2((a,b), (c,d)) = \overline{\beta(b,cd^{-1})}.\]
Then, for any $H$-cocentral closed subgroup $T_0 \subset T$, we have Hopf $*$-algebra isomorphisms
\begin{align*}
{}_{J_1}((\C[T] \otimes H)^{\Delta(T_0)})_{J_1{}^{-1}} &\cong \C[T/T_0] \# H_\beta,\quad a \otimes x \mapsto a\pi(x_{(1)}) \# x_{(2)}, \\
{}_{J_{2}}(\C[T/T_0] \ltimes H)_{J_2{}^{-1}} &\cong \C[T/T_0] \# H_\beta, \quad a\otimes x\mapsto a\# x.
\end{align*}
\end{theorem}

More pedantically, by restriction $J_1$ defines a $2$-cocycle on $\Delta(T_0)^\perp\subset\hat T\times\hat T$, hence a Hopf $2$-cocycle on $\C[(T\times T)/\Delta(T_0)]$. The latter gives rise to a Hopf $2$-cocycle on $(\C[T] \otimes H)^{\Delta(T_0)}$ using the map $\iota\otimes\pi\colon (\C[T] \otimes H)^{\Delta(T_0)}\to\C[(T\times T)/\Delta(T_0)]$. This is the cocycle we use to define ${}_{J_1}((\C[T] \otimes H)^{\Delta(T_0)})_{J_1{}^{-1}}$. Similarly, $J_2$ defines a Hopf $2$-cocycle on $\C[T/T_0] \ltimes H$.

\begin{proof}[Proof of Theorem~\ref{thm:main}.]
%The result can of course be proved by a direct computation, but it is quicker to deduce it from Theorem~\ref{thm:cocycle-transmutation}.
We note that ${}_{J_1}(\C[T] \otimes H)_{{J_1}{}^{-1}}= \C[T] \bowtie_\beta H$, where the latter Hopf algebra is defined as in Theorem~\ref{thm:cocycle-transmutation}. By that theorem, this immediately gives the first Hopf algebra isomorphism for trivial $T_0$. The isomorphism is readily verified to be $*$-preserving, using that by~\eqref{eq:braided-star} and Definition~\ref{def:bos} the involution on $\C[T] \# H_\beta$ is given by
$$
(a\# x)^*=\beta(b^{-1}c,a)\,a^{-1}\# x^{*_\beta}=\beta(b^{-1}c,ab^{-1})\,a^{-1}\# x^*,\quad a\in\hat T,\ x\in H_{b,c},
$$
while by Proposition~\ref{prop:cocycle-twist} the involution on ${}_{J_1}(\C[T] \otimes H)_{J_1{}^{-1}}$ is given by
$$
(a\otimes x)^{*_{J_1}}=\beta(b^{-1}c,a)\,a^{-1}\otimes x^*,\quad a\in\hat T,\ x\in H_{b,c}.
$$

For the same reason as in Proposition~\ref{prop:main-untwisted},  for every $H$-cocentral closed subgroup $T_0\subset T$, the isomorphism  ${}_{J_1}(\C[T] \otimes H)_{J_1{}^{-1}} \cong \C[T] \# H_\beta$ defines by restriction the first isomorphism in the formulation of the theorem. The second isomorphism follows from this and again Proposition~\ref{prop:main-untwisted}.
\end{proof}

It is well-known that $2$-cocycle twisting preserves the monoidal categories of comodules. We thus get the following:
\begin{corollary}
\label{cor:monoidal}
For any $H$-cocentral closed subgroup $T_0 \subset T$, the category $\mathcal{M}^{\C[T/T_0] \# H_\beta}$ is monoidally equivalent to $\mathcal{M}^{\C[T/T_0] \ltimes H}$.
\end{corollary}

\begin{remark}\label{rem:monoidal-equivalence}
The cocentral homomorphism $q\pi\colon H\to\C[T_0]$ defines a $\hat T_0$-grading on the category~$\mathcal{M}^H_f$ of finite dimensional comodules. Then $\mathcal{M}^{\C[T]\otimes H}_f\cong\operatorname{Vect}_f^{\hat T}\boxtimes\mathcal{M}^H_f$, where $\operatorname{Vect}_f^{\hat T}$ is the category of finite dimensional $\hat T$-graded vector spaces, is bi-graded by $\hat T\times\hat T_0$. We can conclude that $\mathcal{M}^{\C[T/T_0] \# H_\beta}_f$ is monoidally equivalent to the subcategory of $\operatorname{Vect}_f^{\hat T}\boxtimes\mathcal{M}^H_f$ generated by the homogeneous components of bi-degree $(a,b)$ such that $q(a)b=1$.\ee
\end{remark}

From Theorems~\ref{thm:cocycle-transmutation} or~\ref{thm:main} we see that we have a Hopf $*$-algebra inclusion
\begin{equation}
\label{eq:H-in-bos}
\phi\colon H \to \C[T] \# H_\beta, \quad \phi(x) = \pi(x_{(1)})\# x_{(2)}.
\end{equation}
This map induces a monoidal functor from the category $H$-comodules to the category of $H_\beta$-comodules. It will be convenient to have the following description of this functor.

\begin{lemma}
\label{lem:braided-reps}
Let $\delta\colon M \to M \otimes H$ be an $H$-comodule. Then $(M, (\iota\otimes \pi)\delta, \delta)$ is an $H_\beta$-comodule.
\end{lemma}
\begin{proof}
Write $\delta' = (\iota\otimes\phi)\delta$. Then
\[(\iota \otimes (\iota\# \varepsilon))\delta' = (\iota\otimes(\pi\otimes \varepsilon)\Delta)\delta = (\iota\otimes\pi)\delta, \]
and
\[ (\iota\otimes (\varepsilon\# \iota))\delta' = (\iota\otimes(\varepsilon\pi\otimes \iota)\Delta)\delta = \delta.\]
Hence, by Proposition \ref{prop:braided-comodule}, the claim follows.
\end{proof}

\subsection{Another view on \texorpdfstring{$H_\beta$}{Hb}}
Motivated by the recent of work of Bochniak and Sitarz~\cite{BS}, we now give another interpretation of the structure maps for $H_\beta$.

Using the left and right coactions of $\C[T]$ on $H$, we can view $H$ as a $\C[T\times T]$-comodule algebra. Then the new product $\cdot_\beta$ on $H_\beta$ is obtained by cocycle twisting (see~\cite[Section~2.3]{Majid}) the original product by the $2$-cocycle
$$
(\hat{T}\times \hat{T}) \times (\hat{T}\times \hat{T}) \to \T,\quad((a,b),(c,d))\mapsto \beta(a^{-1}b,c^{-1}).
$$

On the other hand, $H$ is also a $\C[T]$-comodule coalgebra, so its coalgebra structure can be twisted by a $2$-cocycle $\omega\in Z^2(\hat T;\T)$:
$$
\Delta_\omega(x)=\sum_c\omega(a^{-1}c,c^{-1}b)x^{a,c}_{(1)}\otimes x^{c,b}_{(2)},\quad x\in H_{a,b}.
$$
As a consequence of the following lemma, this always gives an isomorphic comodule coalgebra.

\begin{lemma}
Assume $\omega\in Z^2(\hat T;\T)$ is a normalized cocycle (so $\omega(1,1)=1$) and $\gamma\colon\hat T\times\hat T\to\T$ is a function. Then the identity
\begin{equation}\label{eq:cobound1}
\omega(a^{-1}c,c^{-1}b)=\gamma(a,c)\gamma(c,b)\gamma(a,b)^{-1}
\end{equation}
holds for all $a,b,c\in\hat T$ if and only if
\begin{equation}\label{eq:cobound2}
\gamma(a,b)=\omega(a^{-1},b)^{-1}f(a)g(b)
\end{equation}
for all $a,b\in\T$, where $f,g\colon\hat T\to\T$ are arbitrary functions such that
$$
f(a)g(a)=\omega(a,a^{-1}),\quad a\in \hat T.
$$
\end{lemma}

\bp
Assume~\eqref{eq:cobound1} holds. Letting $c=1$ we get
$$
\omega(a^{-1},b)=\gamma(a,1)\gamma(1,b)\gamma(a,b)^{-1}.
$$
Therefore~\eqref{eq:cobound2} holds with $f(a)=\gamma(a,1)$ and $g(b)=\gamma(1,b)$. The identity $f(a)g(a)=\omega(a,a^{-1})$ is satisfied by~\eqref{eq:cobound2}, since by letting $a=b=c$ in~\eqref{eq:cobound1} we see that $\gamma(a,a)=1$ (recall also that $\omega(a,a^{-1})=\omega(a^{-1},a)$ by the cocycle identity).

Conversely, assume~\eqref{eq:cobound2} holds for some functions $f,g\colon\hat T\to\T$. Then
$$
\gamma(a,c)\gamma(c,b)\gamma(a,b)^{-1}=\omega(a^{-1},c)^{-1}\omega(c^{-1},b)^{-1}\omega(a^{-1},b)f(c)g(c).
$$
Therefore~\eqref{eq:cobound1} holds if and only if
$$
\omega(a^{-1},c)\omega(a^{-1}c,c^{-1}b)\omega(c^{-1},b)=\omega(a^{-1},b)f(c)g(c).
$$
Using the cocycle identity twice, the left hand side equals
$$
\omega(a^{-1},b)\omega(c,c^{-1}b)\omega(c^{-1},b)=\omega(a^{-1},b)\omega(c,c^{-1})\omega(1,b)=\omega(a^{-1},b)\omega(c,c^{-1}),
$$
where the last identity holds, since $\omega$ is normalized. Thus, identity~\eqref{eq:cobound1} holds if and only if $f(c)g(c)=\omega(c,c^{-1})$.
\ep

From this we see that up to an isomorphism $H_\beta$ can be obtained in many different ways by simultaneously twisting the product and coproduct on $H$:

\begin{proposition}
Given a bicharacter $\beta\colon\hat T\times\hat T\to\T$ and a normalized $2$-cocycle $\omega\in Z^2(\hat T;\T)$, choose a function $\gamma\colon\hat T\times\hat T\to\T$ satisfying~\eqref{eq:cobound1} and define a $2$-cocycle $\Omega\in Z^2(\hat T\times\hat T;\T)$ by
$$
\Omega((a,b),(c,d))=\beta(a^{-1}b,c^{-1})\gamma(a,b)^{-1}\gamma(c,d)^{-1}\gamma(ac,bd).
$$
Then we can define new product, coproduct and involution on the $\C[T]$-comodule $H$ by
$$
x\cdot_\Omega y=\Omega((a,b),(c,d))xy,\quad x\in H_{a,b},\ y\in H_{c,d},
$$
$$
\Delta_\omega(x)=\sum_c\omega(a^{-1}c,c^{-1}b)x^{a,c}_{(1)}\otimes x^{c,b}_{(2)},\quad x\in H_{a,b},
$$
$$
x^\star=\overline{\Omega((a,b),(a,b)^{-1})}x^*=\beta(a^{-1}b,a^{-1})\gamma(a,b)\gamma(a^{-1},b^{-1})x^*,\quad x\in H_{a,b},
$$
to get a braided Hopf $*$-algebra $H_{\Omega,\omega}\in\hopf(\C[T],\beta)$. We have an isomorphism
$$
H_\beta\cong H_{\Omega,\omega},\quad H_{a,b}\ni x\mapsto \gamma(a,b)x.
$$
\end{proposition}

\begin{example}
Consider $T=\T$ and, identifying $\hat\T$ with $\Z$, let
$$
\beta(m,n)=e^{2i\phi mn},\qquad \omega(m,n)=e^{i\phi mn}
$$
for some $\phi\in\R$. By taking $\gamma(m,n)=\omega(-m,n)^{-1}e^{-i\phi m^2}=e^{i\phi(mn-m^2)}$, we get
$$
\Omega((k,l),(m,n))=e^{i\phi(kn-lm)}.
$$
Then $H_{\Omega,\omega}$ coincides with the braided Hopf algebra defined in~\cite{BS}.
\end{example}

\subsection{Braided compact matrix quantum groups}
In our examples we will mainly be interested in transmutations of compact quantum groups. A compact quantum group $G$ is a Hopf $*$-algebra $\C[G]$ that is spanned (equivalently, generated as an algebra) by matrix coefficients of finite dimensional unitary comodules. We refer the reader to \cite{NT13} for an introduction to the subject and we will often use the terminology there. For instance, a $\C[G]$-comodule will sometimes be called a representation of~$G$.

\smallskip

Recall that fixing a basis in the underlying vector space of an $m$-dimensional $H$-comodule defines a \textit{corepresentation matrix} for $H$. This is a matrix $U = (u_{ij})_{i,j} \in\Mat_m(H)$ such that
\begin{equation}
\label{eq:corepmatrix}
\Delta(u_{ij}) = \sum_k u_{ik} \otimes u_{kj}, \quad \varepsilon(u_{ij}) = \delta_{ij}.
\end{equation}
Conversely any such matrix defines an $m$-dimensional $H$-comodule $\delta_U\colon M \to M \otimes H$, by setting
\begin{equation}
\label{eq:delta_U}
\delta_U(e_j) = \sum_i e_i \otimes u_{ij}
\end{equation}
for a fixed vector space $M$ with basis $(e_i)_i$. If $U$ is unitary, then the conjugate corepresentation matrix is $\bar{U} = (u_{ij}^*)_{i,j}$.

\begin{definition}[\cite{Wor1,Wor2}]
\label{def:matrix}
A \emph{compact matrix quantum group} $G$ is a Hopf $*$-algebra $\C[G]$ with generators $u_{ij}$, $1 \leq i,j \leq m$, such that
\begin{itemize}
\item[(i)]  $U = (u_{ij})_{i,j}$ is a unitary corepresentation matrix;
\item[(ii)] $\bar{U} = (u_{ij}^*)_{i,j}$ is equivalent to a unitary corepresentation matrix.
\end{itemize}
The coproduct $\Delta$, counit $\varepsilon$ and antipode $S$ are then given by
\[ \Delta(u_{ij}) = \sum_k u_{ik} \otimes u_{kj}, \quad \varepsilon(u_{ij}) = \delta_{ij}, \quad S(u_{ij}) = u_{ji}^*. \]
The matrix $U$ is called the \emph{fundamental unitary} for the compact matrix quantum group.
\end{definition}

We remark that this is not the original definition of Woronowicz, but it is equivalent to that by a result of Dijkhuizen and Koornwinder~\cite{DK}, see also~\cite[Section~1.6]{NT13}.

\smallskip

Next, we want to introduce a braided analogue of this definition, but first we need some preparation. Suppose that $A \in \hopf(K,R)$ for a Hopf $*$-algebra $K$ with a unitary coquasitriangular structure $R$. Suppose $U=(u_{ij})_{i,j} \in \Mat_m(A)$ satisfies $\Delta_A(u_{ij})=\sum_ku_{ik}\otimes u_{kj}$, $\eps_A(u_{ij})=\delta_{ij}$. We can still define $\delta_U\colon M \to M \otimes A$ as in~\eqref{eq:delta_U} and this gives a comodule for the coalgebra~$A$. By definition, if $Z=(z_{ij})_{i,j} \in \Mat_m(K)$ is a corepresentation matrix, the triple $(M, \delta_Z, \delta_U)$ defines an $A$-comodule if and only if
\[ (\delta_U\otimes\iota) \circ \delta_Z = \delta_{M\otimes A} \circ \delta_U,\]
where $\delta_{M\otimes A}$ denotes the tensor product comodule in $\mathcal{M}^K$. We have the following characterization:
\begin{lemma}
\label{lem:conjugate}
In the above setting, the pair $(Z,U)$ defines an $A$-comodule if and only if
\begin{equation}
\label{eq:condition}
\delta(u_{ij}) = \sum_{s,t} u_{st} \otimes S(z_{is})z_{tj},
\end{equation}
where $\delta\colon A \to A \otimes K$ is the coaction by $K$ on $A$. Furthermore, the $A$-comodule we thus get is unitary (that is, the corresponding $(K\#A)$-comodule is unitary) if and only if $U$ and $Z$ are unitary, and then the conjugate comodule is given by the pair $(\bar{Z}, \bar{U}_Z)$, where
\begin{equation}\label{eq:barUZ}
\bar{U}_Z = (\bar{u}_{ij}^Z)_{i,j}, \quad \bar{u}_{ij}^Z = \sum_{s,l,t} R(z_{tj}^*z_{sl}, z_{il}^*)u_{st}^*,
\end{equation}
while the antipode on $A$ satisfies $S_A(u_{ij})=u^*_{ji}$.
\end{lemma}

Here by the conjugate $A$-comodule we mean the comodule obtained by taking the conjugate ($K\#A)$-comodule.

\begin{proof}
The condition $(\delta_U\otimes\iota) \circ \delta_Z = \delta_{M\otimes A} \circ \delta_U$ is satisfied if and only if
\[ \sum_t u_{st} \otimes z_{tj} = \sum_k u_{kj}^{(1)} \otimes z_{sk}u_{kj}^{(2)} \]
for $1 \leq s,j \leq m$, where we write $\delta(a) = a^{(1)} \otimes a^{(2)}$, $a \in A$. Multiplying both sides by $S(z_{is})$ and summing over $s$ yields
\[ \sum_{s,t} u_{st} \otimes S(z_{is})z_{tj} = \sum_{k,s} u_{kj}^{(1)} \otimes S(z_{is})z_{sk}u_{kj}^{(2)} = u_{ij}^{(1)} \otimes u_{ij}^{(2)}. \]
This implies the first statement.

For the second one, note that the $A$-comodule defined by $(Z,U)$ corresponds to the $(K \# A)$-comodule given by the matrix
\[ W =(w_{ij})_{i,j}= Z\# U, \quad \text{so}\quad w_{ij} = \sum_k z_{ik} \# u_{kj}.\]
Using that  $\iota\#\eps_A\colon K\#A\to K$ is a $*$-homomorphism, it is easy to see that $W$ is unitary if and only $Z$ and $U$ are unitary.

When $U$ is unitary, the equality $S_A(u_{ij}) = u_{ji}^*$ holds by the antipode identity. The claim about the conjugate comodule follows from the $*$-structure on $K \# A$ and the fact that $\bar{U}_Z = (\varepsilon_K\#\iota)(\bar{W})$. Alternatively, we can use the formula for the antipode in Definition \ref{def:bos} to get
\[ \bar{u}^Z_{ij} = (\varepsilon_K \# \iota)S(w_{ji}) = \sum_r R( S_A(u_{jr})^{(2)}, S_{K}(z_{ri}))S_A(u_{jr})^{(1)}. \]
As $S_A(u_{jr}) = u_{rj}^*$, we recover~\eqref{eq:barUZ}.
\end{proof}

We remark that it is important to keep track of both $R$ and $Z$ in the definition of $\bar{U}_Z$ above. However, we stick to the notation $\bar{U}_Z$ for the rest of the paper, as $R$ will always be given by a fixed bicharacter $\beta$.

\begin{definition}\label{def:braided-matrix}
Let $T$ be a compact abelian group with a fixed unitary corepresentation matrix $Z \in \Mat_m(\C[T])$ and a bicharacter $\beta$ on $\hat{T}$. A \emph{braided compact matrix quantum group} over the triple $(T,Z, \beta)$ is an object $A \in \hopf(\C[T],\beta)$ generated as a $*$-algebra by elements~$u_{ij}$, $1 \leq i,j \leq m$, such that, for $U = (u_{ij})_{i,j}$,
\begin{itemize}
\item[(i)] $(Z,U)$ defines a unitary $A$-comodule;
\item[(ii)] $(\bar{Z}, \bar{U}_Z)$ defines a unitarizable $A$-comodule.
\end{itemize}
We say that $U$ is the \emph{fundamental unitary} for $A$, while the pair $(Z,U)$ is the \emph{fundamental unitary representation}.
\end{definition}

More explicitly, by Lemma~\ref{lem:conjugate}, conditions (i) and (ii) mean that $U\in\Mat_m(A)$ is unitary, there is $F\in\operatorname{GL}_m(\C)$ such that both $F\bar ZF^{-1}$ and $F\bar U_Z F^{-1}$ are unitary, and the structure maps for $A$ satisfy the following properties: the coaction of $\C[T]$ on $A$ is given by~\eqref{eq:condition}, and
$$
\Delta_A(u_{ij})=\sum_ku_{ik}\otimes u_{kj},\qquad \eps_A(u_{ij})=\delta_{ij},\qquad S_A(u_{ij})=u_{ji}^*.
$$

We remark that in view of Lemma~\ref{lem:conjugate} we can similarly  define a braided compact matrix quantum group over $(K,Z,R)$ for any Hopf $*$-algebra $K$ with a unitary coquasitriangular structure $R$ and a unitary corepresentation matrix $Z=(z_{ij})_{i,j} \in \Mat_m(K)$, but the adjective ``compact'' in this generality might be somewhat misleading.

\begin{proposition}\label{prop:bosonizationBMQG}
Given a compact abelian group $T$, a unitary corepresentation matrix $Z \in \Mat_m(\C[T])$ and a bicharacter $\beta$ on $\hat{T}$, the bosonization of any braided compact matrix quantum group~$A$ over $(T,Z, \beta)$ is a compact quantum group.
\end{proposition}

\bp
By working in an orthonormal basis where $Z$ is diagonal, we see that the $*$-algebra $\C[T]\# A$ is generated by the matrix coefficients of the fundamental unitary representation and the characters of $T$.
\ep

\begin{proposition} \label{prop:MQG-trans}
Let $G$ be a compact matrix quantum group with fundamental unitary $U = (u_{ij})^m_{i,j=1}$. Assume that $T$ is a compact abelian group with a Hopf $*$-algebra map $\pi\colon\C[G] \to \C[T]$, and let $\beta$ be a bicharacter on $\hat{T}$. Then the transmutation $\C[G]_\beta$ is a braided compact matrix quantum group over $(T, \pi(U), \beta)$ with fundamental unitary $U$.
\end{proposition}
\begin{proof}
Put $Z=\pi(U)$. Recall that by~\eqref{eq:H-in-bos} we have a Hopf $*$-algebra map
$$
\phi\colon \C[G] \to \C[T] \# \C[G]_\beta,\quad \phi(x) = \pi(x_{(1)}) \# x_{(2)}.
$$
By Lemma \ref{lem:braided-reps} this implies that the pair $(Z, U)$ defines a unitary $\C[G]_\beta$–co\-mo\-du\-le, with the corresponding $(\C[T]\# \C[G]_\beta)$-comodule given by the unitary $\phi(U)=Z\# U$. The conjugate $(\C[T]\# \C[G]_\beta)$-comodule is given by $\phi(\bar U)$. As $\bar U\in\Mat_m(\C[G])$ is unitarizable, by Lemma~\ref{lem:conjugate} we see that both conditions (i) and (ii) in Definition~\ref{def:braided-matrix} are satisfied.

It remains to check that $\C[G]_\beta$ is generated by the matrix coefficients of $U$ as a $*$-algebra. This becomes clear if we work in an orthonormal basis where $Z$ is diagonal, as then the products of the elements $u_{ij}$ and their adjoints in $\C[G]$ and $\C[G]_\beta$ coincide up to phase factors.
\end{proof}

\begin{remark}
\label{rem:other-fund}
Even though $\C[G]_\beta$ has fundamental representation $(\pi(U),U)$, condition~\eqref{eq:condition} can be satisfied for another pair $(Z',U)$, which is then also a fundamental representation. A particularly interesting situation is when $Z' \in \Mat_m(\C[T/T_0])$ for a $\C[G]$-cocentral subgroup $T_0 \subset T$. In this case we can view $\C[G]_\beta$ as a braided compact matrix quantum group over $(T/T_0, Z', i^*\beta)$. \ee
\end{remark}

We record a useful lemma related to the above remark.

\begin{lemma}
\label{lem:wZ}
Assume $A \in \hopf(\C[T],\beta)$ and take $w\in\hat T$. Then $(Z,U)$ defines an $A$-comodule if and only if $(w Z, U)$ defines an $A$-comodule. If in addition $Z$ and $U$ are unitary, then we have the relation
\[ \bar{U}_{w Z} = D \bar{U}_{Z} D^{-1}, \quad D = (\beta(z_{ij}^*, w))_{i,j}. \]
\end{lemma}
\begin{proof}
The first claim is obvious from condition~\eqref{eq:condition}. The second claim is easy to check in an orthonormal basis where $Z$ is diagonal, in which case it follows immediately from~\eqref{eq:barUZ}. More conceptually, one can check that for the corepresentation matrix $W=(\sum_k z_{ik} \# u_{kj})_{i,j}$ for $\C[T]\#A$ we have
$$
(w\#1)W=XW(w\#1)X^{-1},
$$
where $X=(\beta(z_{ij},w))_{i,j}$, which is a matrix commuting with $Z$. This implies that $\bar U_{wZ}=\bar X\bar U_Z\bar X^{-1}$. It remains to observe that $\overline{\beta(x, w)}=\beta(x^*,w)$ for all $x\in\C[T]$ to see that $\bar X=D$.
\end{proof}

\bigskip

\section{Examples: transmuting matrix quantum groups}\label{s3}

Before we embark on the examples, we remark that in a number of recent papers (see, e.g., \cite{KMRW,MR-ort,BJR, ABRR}) braided quantum groups are constructed in a C$^*$-algebraic setting. However, the corresponding bosonizations are C$^*$-algebraic compact quantum groups, and these always have dense $*$-subalgebras of matrix coefficients, which leads to purely algebraic results. Conversely, in our examples the bosonizations will be compact quantum groups by Theorem~\ref{thm:main} (as unitary cocycle twisting preserves compactness) or Proposition~\ref{prop:bosonizationBMQG}, and hence they can be completed to C$^*$-algebraic compact quantum groups. We can therefore go back and forth between the $*$-algebraic and C$^*$-algebraic settings. Below we will not dwell on the specific details of this but rather stick to the algebraic picture.

\subsection{Braided \texorpdfstring{$SU_q(2)$}{SUq(2)}}
\label{ex:SU_q(2)}
Fix $q > 0$ and recall that $H := \C[SU_q(2)]$ is the universal unital $*$-algebra with generators $\alpha$ and $\gamma$ subject to the relations
\[\alpha\gamma = q\gamma\alpha, \quad \alpha\gamma^* = q \gamma^*\alpha, \quad \gamma^*\gamma = \gamma\gamma^*,\]
\[ \quad \alpha^*\alpha + \gamma^*\gamma = 1, \quad \alpha\alpha^* + q^2\gamma\gamma^* = 1.\]
It is a Hopf $*$-algebra with coproduct
\[\Delta(\alpha) = \alpha \otimes \alpha - q\gamma^* \otimes \gamma, \quad \Delta(\gamma) = \gamma \otimes \alpha+\alpha^* \otimes \gamma.\]

Consider the map
\[ \pi\colon H \to \C[\T] = \C[z,z^{-1}], \quad \pi(\alpha) = z, \quad \pi(\gamma) = 0.\]
Under the identification $\Z=\hat \T$, we have
$$
\alpha\in H_{1,1},\qquad \gamma\in H_{{-1},1},
$$
and the restricted right adjoint coaction $\mathrm{ad}_\pi$ is determined by
\begin{equation*}
\label{eq:SU(2)-adjoint-action}
\mathrm{ad}_\pi(\alpha) = \alpha \otimes 1, \quad \mathrm{ad}_\pi(\gamma) = \gamma \otimes z^{2}.
\end{equation*}

For $\lambda \in \T$, define a bicharacter on $\Z=\hat \T$ by $\beta_\lambda(m, n) = \lambda^{-mn}$. To find relations in the transmutation $H_\lambda = H_{\beta_\lambda}$ we  write $a\cdot b=a\cdot_{\beta_\lambda} b$ and $a^{*_\lambda} = a^{*_{\beta_\lambda}}$. Then, by~\eqref{eq:trans-prod} and~\eqref{eq:trans-star},
$$
\alpha^{*_\lambda}=\beta_\lambda(0,{-1})\alpha^*=\alpha^*,\qquad \gamma^{*_\lambda}=\beta_\lambda(2,1)\gamma^*=\lambda^{-2}\gamma^*,
$$
and
$$
\alpha \cdot \gamma = \beta_\lambda(0,1) \alpha\gamma=\alpha\gamma = q\gamma\alpha = q \beta_\lambda(2,{-1})^{-1} \gamma \cdot \alpha = q\lambda^{-2} \gamma \cdot \alpha,
$$
$$
\alpha \cdot \gamma^{*_\lambda} = \beta_\lambda(0,{-1}) \alpha\gamma^{*_\lambda}= \alpha\gamma^{*_\lambda} = q\gamma^{*_\lambda}\alpha = q \beta_\lambda({-2},{-1})^{-1} \gamma^{*_\lambda} \cdot \alpha = q\lambda^{2} \gamma^{*_\lambda} \cdot \alpha,
$$
$$
\gamma^{*_\lambda} \cdot \gamma = \beta_\lambda({-2},1) \gamma^{*_\lambda}\gamma =\gamma^*\gamma=\gamma\gamma^* = \beta_\lambda(2,{-1})\gamma\gamma^{*_\lambda} = \gamma \cdot \gamma^{*_\lambda},
$$
$$
\alpha \cdot \alpha^{*_\lambda} = \alpha\alpha^*,\qquad \alpha^*\alpha = \alpha\cdot\alpha^{*_\lambda}.
$$
Defining $q' = q\lambda^{2}$ we get the following relations in $H_\lambda$: \[\alpha\cdot\gamma = \bar{q}\, '\gamma\cdot\alpha, \quad \alpha\cdot\gamma^{*_\lambda} = q' \gamma^{*_\lambda}\cdot\alpha, \quad \gamma^{*_\lambda}\cdot\gamma = \gamma\cdot\gamma^{*_\lambda},\]
\[ \quad \alpha^{*_\lambda}\cdot\alpha + \gamma^{*_\lambda}\cdot\gamma = 1, \quad \alpha\cdot\alpha^{*_\lambda} + |q'|^2\gamma\cdot\gamma^{*_\lambda} = 1.\]
It is not difficult to see that these relations completely describe the transmuted algebra; in the next subsection we will prove a more general result.
The coproduct remains unchanged, so we have
\[\Delta(\alpha) = \alpha \otimes \alpha - q'\gamma^{*_\lambda} \otimes \gamma, \qquad \Delta(\gamma) = \gamma \otimes \alpha+\alpha^{*_\lambda} \otimes \gamma.\]

These formulas are the same as for the braided quantum group $SU_{q'}(2)$ constructed in \cite{KMRW}, modulo a small but important nuance. By Proposition~\ref{prop:braided-over-quotient}, $H_\lambda$ can be viewed as a braided quantum group over different tori. Namely, we see that $H_\lambda$ can be viewed as a braided compact matrix quantum group over both triples
\begin{equation*}
\label{eq:2repsSUq(2)}
(\T, \, \begin{pmatrix}
z & 0 \\
0 & z^{-1}
\end{pmatrix} , \, \beta_\lambda) \quad \mbox{ and } \quad  (\T/T_0, \, \begin{pmatrix}
z^2 & 0 \\
0 & 1
\end{pmatrix}, \, i^*\beta_\lambda),
\end{equation*}
where $T_0 = \{-1,1\} \subset \T$, see Remark \ref{rem:other-fund}. As $\C[\T/T_0]$ is generated by $z^2$ we have the isomorphism
\begin{equation*}
\label{eq:tori-iso}
f\colon \C[\mathbb{T}/T_0] \to \C[w,w^{-1}], \quad f(z^2) = w.
\end{equation*}
Moreover, $(i \circ f^{-1})^*\beta_\lambda = \beta_{\zeta}$, where $\zeta = \lambda^4 = q'/\bar{q}\,'$. Therefore we can consider $H_\lambda$ as a braided compact matrix quantum group over the triple
\[(\T, \, \begin{pmatrix}
w & 0 \\
0 & 1
\end{pmatrix}, \, \beta_\zeta).\]
This is the braided quantum group $\C[SU_{q'}(2)] \in \hopf(\C[\T], \beta_\zeta)$ considered in \cite{KMRW}.

\medskip

Finally, let us consider the bosonizations. By Theorem~\ref{thm:main} we have
\[\C[z,z^{-1}]\# H_\lambda \cong {}_{J}(\C[z,z^{-1}] \otimes \C[SU_q(2)])_{J^{-1}},\]
where $J((m,n),({m'},{n'})) = \lambda^{nm'}$. In other words, the bosonization is a cocycle twist of the compact quantum group $\T\times SU_q(2)$.

On the other hand, by the same theorem, the bosonization $\C[w,w^{-1}] \, \# \, H_\lambda$ is a cocycle twist of
$(\T\times SU_q(2))/\Delta(T_0)$. It is easy to see that the latter quantum group is isomorphic to~$U_q(2)$, similarly to the classical isomorphism
\begin{equation*}\label{eq:U(2)}
(\T\times SU(2))/{\Delta(T_0)}\cong U(2),\quad [(z,U)]\mapsto\begin{pmatrix}
z & 0 \\
0 & z
\end{pmatrix}U,
\end{equation*}
and therefore its cocycle twist must be one of the quantum deformations of~$U(2)$ studied in~\cite{ZZ}, cf.~\cite{KMRW}.

\subsection{Braided free orthogonal quantum groups}
Let $m \geq 2$ be a natural number, $F \in \mathrm{GL}_m(\C)$ and assume that $F\bar{F} = \pm 1$. Let $\C[O_F^+]$ be the universal unital $*$-algebra generated by elements~$u_{ij}$, $1 \leq i,j \leq m$, subject to the relations
\[U = (u_{ij})_{i,j}\quad \mbox{is unitary and}\quad U = F\bar{U}F^{-1}.\]
The Hopf $*$-algebra structure on $\C[O_F^+]$ is defined as in Definition~\ref{def:matrix}, and $O_F^+$ is called a \textit{free orthogonal quantum group}.

Let $T$ be a compact abelian group and $Z\in\Mat_m(\C[T])$ be a unitary corepresentation matrix satisfying $F\bar{Z}F^{-1} = Z$. By the universality of $\C[O_F^+]$ there is a Hopf $*$-algebra map $\pi\colon \C[O_F^+] \to \C[T]$ such that $\pi(u_{ij})=z_{ij}$. Fix a bicharacter $\beta\colon \hat{T} \times \hat{T} \to \T$ and consider the transmutation $\C[O_F^+]_\beta$. It is natural to call it a \emph{braided free orthogonal quantum group}.

\begin{proposition}
\label{prop:O_F^+-universal}
The braided Hopf $*$-algebra $\C[O_F^+]_\beta$ is a braided compact matrix quantum group over $(T,Z,\beta)$ with fundamental unitary $U=(u_{ij})_{i,j}$.
As a $*$-algebra, it is a universal unital $*$-algebra with generators~$u_{ij}$ satisfying the relations
\begin{equation}
\label{eq:braided-free-ort}
U = (u_{ij})_{i,j} \quad\mbox{is unitary and}\quad U = F\bar{U}_Z F^{-1},
\end{equation}
where $\bar{U}_Z= (\bar{u}_{ij}^Z)_{i,j}$ and $\bar{u}_{ij}^Z = \sum_{s,l,t} \beta(z_{tj}^*z_{sl}, z_{il}^*)u_{st}^*$.
\end{proposition}

\bp
For the purpose of this proof let us denote the fundamental unitary of $O^+_F$ by $V=(v_{ij})_{i,j}$ and write $A$ for $\C[O_F^+]_\beta$. The first claim follows from Proposition~\ref{prop:MQG-trans}. Relations~\eqref{eq:braided-free-ort} are obtained by considering, as in the proof of that proposition, the Hopf $*$-algebra map $\phi\colon \C[O^+_F]\to\C[T]\# A$, $\phi(x)=\pi(x_{(1)})\# x_{(2)}$, and using that $\phi(V)=Z\# U$, $\phi(\bar V)=\bar Z\# \bar U_Z$.

\iffalse{
For the purpose of this proof let us denote the fundamental unitary of $O^+_F$ by $V=(v_{ij})_{i,j}$ and write $A$ for $\C[O_F^+]_\beta$. By~\eqref{eq:H-in-bos}, we have a canonical Hopf $*$-algebra homomorphism $\phi\colon \C[O^+_F]\to\C[T]\# A$, $\phi(x)=\pi(x_{(1)})\# x_{(2)}$. Then, by definition, for $W=\phi(V)$ we get
$$
(\iota\#\eps_A)(W)=Z,\qquad (\eps_{\C[T]}\#1)(W)=U,
$$
since under the identification of the coalgebras $\C[O^+_F]$ and $A$ we have $V=U$. By Proposition~\ref{prop:braided-comodule} it follows that $U$ is a corepresentation matrix for $A$ and $W=Z\#U$. By Lemma~\ref{lem:conjugate} we can conclude then that $U$ is unitary and $(\eps_{\C[T]}\#\iota)(\bar W)=\bar U_Z$. As $W=F\bar W F^{-1}$, it follows also that $U=F\bar U_ZF^{-1}$.
\fi

\smallskip

It remains to show that as a $*$-algebra $A$ is completely described by relations~\eqref{eq:braided-free-ort}. Consider a universal unital $*$-algebra $\tilde A$ with generators $\tilde u_{ij}$ satisfying these relations, and let $\rho\colon\tilde A\to A$ be the $*$-homomorphism such that $\rho(\tilde u_{ij})=u_{ij}$.

Working in a basis where~$Z$ is diagonal, it is not difficult to check that $\tilde A$ is a $\C[T]$-comodule $*$-algebra, with the coaction of~$\C[T]$ given by
$$
\tilde\delta(\tilde u_{ij}) = \sum_{s,t} \tilde u_{st} \otimes z_{si}^*z_{tj},\quad\text{or}\quad (\iota\otimes\tilde\delta)(U)=Z^*_{13}\tilde U_{12}Z_{13}.
$$

Consider the smash product $\C[T]\# \tilde A=\C[T]\otimes_\beta\tilde A$. It is again not difficult to check that we have a $*$-homomorphism
$$
\tilde\phi\colon\C[O^+_F]\to \C[T]\# \tilde A,\quad \tilde\phi(v_{ij})=\sum_kz_{ik}\#\tilde u_{kj}.
$$
Define linear maps
\begin{align*}
&\tilde\psi\colon \C[T]\otimes\C[O^+_F]\to \C[T]\# \tilde A,\qquad \tilde\psi(x\otimes a)=x\tilde\phi(a),\\
&\psi\colon \C[T]\otimes\C[O^+_F]\to \C[T]\# A,\qquad \psi(x\otimes a)=x\pi(a_{(1)})\# a_{(2)}.
\end{align*}
Then $\psi=(\iota\#\rho)\tilde\psi$. The map $\psi$ is a linear isomorphism, e.g., by Theorem~\ref{thm:main}. On the other hand, the map $\tilde\psi$ is surjective, which becomes particularly clear if we work in a basis where~$Z$ is diagonal and therefore $\tilde\phi(v_{ij})=z_{ii}\#\tilde u_{ij}$. (Alternatively, we can observe that $\tilde\psi$ defines a homomorphism $\C[T]\#\C[O^+_F]\to\C[T]\#\tilde A$, cf.~Remark~\ref{rem:cocycle}, and its image contains the elements $1\#\tilde u_{ij}$.) It follows that $\tilde\psi$ is a linear isomorphism and hence~$\rho$ is an isomorphism as well.
\ep

Next, we want to change the perspective on the braided free orthogonal quantum groups and show how they can be associated with a larger class of matrices than $F$ as above.

\begin{proposition}
\label{prop:G_A}
Let $A\in\mathrm{GL}_m(\C)$ ($m\ge2$) be a matrix such that $A\bar A$ is unitary, and choose a sign $\tau=\pm1$, with $\tau=1$ if $m$ is odd. Then there are a compact abelian group $T$, a unitary corepresentation matrix $X=(x_{ij})_{i,j}\in\Mat_m(\C[T])$, a character $w\in\hat T$ and a bicharacter $\beta\colon\hat T\times\hat T\to\T$ such that
\begin{equation}\label{eq:BFO-condition}
A(w^2 \bar{X})A^{-1} = X\quad \mbox{in}\quad \Mat_m(\C[T])
\end{equation}
and $AC\overline{AC} = \tau1$, where $C = (\beta(x_{ij}^*, w))_{i,j}$.
For every such quadruple $(T,X,w,\beta)$, consider a universal unital $*$-algebra $\C[O_A^{X,\beta}]$ with generators~$u_{ij}$ satisfying the relations
\begin{equation*}
\label{eq:braided-free-ort2}
U = (u_{ij})_{i,j} \quad\mbox{is unitary and}\quad U = A\bar{U}_X A^{-1},
\end{equation*}
where $\bar{U}_X= (\bar{u}_{ij}^X)_{i,j}$ and $\bar{u}_{ij}^X = \sum_{s,l,t} \beta(x_{tj}^*x_{sl}, x_{il}^*)u_{st}^*$. Then $\C[O_A^{X,\beta}]$, equipped with the coaction
$$
\delta(u_{ij}) = \sum_{s,t} u_{st} \otimes x_{si}^*x_{tj},
$$
is a braided compact matrix quantum group over $(T,X,\beta)$ with fundamental unitary $U$.
\end{proposition}

\begin{proof}
Assume first that a quadruple $(T,X,w,\beta)$ as in the formulation indeed exists. Define $F =AC$. By our assumptions this matrix satisfies $F\bar{F} = \tau1$ and, as $C$ commutes with $\bar X$, we have
\[F(w\bar{X})F^{-1} = w^{-1}X.\]
By universality there is a Hopf $*$-homomorphism $\pi\colon \C[O_F^+] \to \C[T]$ sending the fundamental representation to $Z = w^{-1}X$. We claim that the corresponding transmutation $\C[O_F^+]_\beta$ satisfies all the required properties of~$\C[O_A^{X,\beta}]$.

Indeed, by Lemma~\ref{lem:wZ}, in $\C[O_F^+]_\beta$ we have
$$
\bar U_X=\bar U_{wZ}=D\bar U_Z D^{-1},
$$
where
$
D=(\beta(z^*_{ij},w))_{i,j}=(\beta(wx^*_{ij},w))_{i,j}=\beta(w,w)C$.
Then $A=\beta(w,w)FD^{-1}$, and the claim follows from Proposition~\ref{prop:O_F^+-universal}.

\smallskip

Next we explain the existence of $(T,X,w,\beta)$. By \cite[Proposition~1.5]{HN21}, we can find a unitary~$v$ such that $vAv^t$ has the form
\begin{equation}\label{eq:TL-matrix}
\begin{pmatrix}
0 & & a_m \\
& \iddots & \\
a_1 & & 0
\end{pmatrix}, \quad a_i\bar{a}_{m-i+1} = \lambda_i \in \T.
\end{equation}
If we can find a quadruple $(T,X,w,\beta)$ for this matrix, then $(T,v^*X(\cdot)v,w,\beta)$ is a quadruple for~$A$. Thus, we may assume that $A$ has the above form.

We will construct $T$ and $X$ such that $X$ is diagonal, so $X(t)=\operatorname{diag}(x_1(t),\dots,x_m(t))$ for some characters $x_i$. The conditions~\eqref{eq:BFO-condition} and $AC\overline{AC}=\tau1$ for $C =\operatorname{diag}(\beta(x_{1}^{-1}, w),\dots,\beta(x_{m}^{-1}, w))$ mean then that
\begin{equation}\label{eq:wTL}
x_ix_{m-i+1}=w^2\qquad\text{and}\qquad \beta(x_i^{-1}x_{m-i+1},w)=\tau\lambda_i.
\end{equation}
If $m=2k$, these conditions can be easily satisfied for the dual $T$ of a free abelian group with independent generators $x_1,\dots,x_k,w$ by letting $x_{m-i+1}=w^2x_i^{-1}$ for $1\le i\le k$. If $m=2k+1$, then $\tau=1$, $\lambda_{k+1}=1$ and the conditions can be satisfied for the dual $T$ of a free abelian group with independent generators $x_1,\dots,x_{k+1}$ by letting $w=x_{k+1}$ and $x_{m-i+1}=w^2x_i^{-1}$ for $1\le i\le k$.
\end{proof}

As is clear from the proof of this proposition, the braided quantum groups $O^{X,\beta}_A$ lie within the class of braided free orthogonal quantum groups that we defined by transmutation. Namely, we have the following:
\begin{corollary}\label{cor:free-orthogonal2}
The braided Hopf $*$-algebra $\C[O_A^{X,\beta}]\in\hopf(\C[T],\beta)$ is isomorphic to the transmutation $\C[O_F^+]_\beta$ with respect to the map $\C[O_F^+]\to\C[T]$, $U\mapsto w^{-1}X$, where $F = AC$.
\end{corollary}

\begin{remark}\label{rem:TL}
A moment's reflection shows that in the proof of Proposition~\ref{prop:G_A} we could take a slightly smaller group $T$ and arrange $X$ to be faithful. Namely, if $m=2k$, instead of taking~$w$ as a separate independent generator, we could let $w=x_1^j$ for any $j\ne0,1$. Similarly, for $m=2k+1$ we could take $x_{k+1}=w=x_1^j$ for any $j\ne0,1$. In both cases we cannot choose groups of a smaller rank in general, since the numbers $\tau\lambda_i$ generate a group of rank up to $k=[m/2]$.
\end{remark}

\begin{remark}\label{rem:BFO}
Once \eqref{eq:BFO-condition} is satisfied, condition $AC\overline{AC} = \tau1$ can be formulated as follows. Let $t_w\in T$ be the element such that $x(t_w)=\beta(x,w)$ for all $x\in\hat T$, so that $C=\overline{X(t_w)}$. Then the requirement is
\begin{equation}\label{eq:BFO-condition2}
A\bar A= \tau \beta(w,w)^{2}X(t_w)^{-2}.
\end{equation}
As a prerequisite for constructing $O^{X,\beta}_A$ this condition can be written as
\begin{equation}\label{eq:BFO-condition2a}
A\bar A= c\, X(t_w)^{-2}\quad\text{for some}\quad c\in\T.
\end{equation}
Indeed, assume \eqref{eq:BFO-condition} and \eqref{eq:BFO-condition2a} are satisfied. Then applying complex conjugation and conjugation by $A$ to the last identity we get
$$
A\bar A= \bar c\,\beta(w,w)^4 X(t_w)^{-2}.
$$
Hence $c=\bar c\,\beta(w,w)^4$ and therefore $\tau:=c\,\beta(w,w)^{-2}=\pm1$, so that~\eqref{eq:BFO-condition2} is satisfied for this $\tau$.  Note that the sign must be~$+1$ for odd $m$, which becomes obvious if we choose a unitary $v$ such that $vA\,\overline{X(t_w)}v^t$ is of the form~\eqref{eq:TL-matrix}.
Thus, the braided compact matrix quantum groups $O^{X,\beta}_A$ are defined under assumptions~\eqref{eq:BFO-condition} and~\eqref{eq:BFO-condition2a}.\ee
\end{remark}

In the setting of Proposition~\ref{prop:O_F^+-universal}, assume now that $T_0\subset T$ is a closed $\C[O_F^+]$-cocentral subgroup. Since the elements $u_{ij}\in\C[O^+_F]$ are linearly independent, this means that the matrices~$Z(t)$, $t\in T_0$, are scalar. Then the condition $Z=F\bar Z F^{-1}$ implies that $Z(t)=\pm1$, so we get a character $\chi\colon T_0\to\{\pm1\}$. When it is nontrivial, it defines the standard $(\Z/2\Z)$-grading on $\Rep O^+_F$.

It is known that the quantum group $O^+_F$ is monoidally equivalent to $SU_q(2)$ for an appropriate~$q$ (see~\cite[Theorem~2.5.11]{NT13}), and this equivalence respects the $(\Z/2\Z)$-gradings. Combining this with Remark~\ref{rem:monoidal-equivalence}, one can conclude that the bosonization of the braided Hopf $*$-algebra  $\C[O^+_F]_\beta\in\hopf(\C[T/T_0],i^*\beta)$ is monoidally equivalent~to
$$
(T\times SU_q(2))/{(\operatorname{id}\times\chi)\Delta(T_0)}.
$$
Together with Corollary~\ref{cor:free-orthogonal2} this leads to the following conclusion.

\begin{proposition}\label{prop:monoidal-orthogonal}
In the setting of Proposition~\ref{prop:G_A}, let $q\in[-1,1]\setminus\{0\}$ be such that $\operatorname{sgn}q=-\tau$ and $|q+q^{-1}|=\Tr(A^*A)$. Assume $T_0\subset T$ is a closed subgroup satisfying $X(t)=\pm w(t)1$ for all $t\in T_0$, and let $\chi\colon T_0\to\{\pm1\}$ be the character such that $X=\chi w1$ on $T_0$. Then the bosonization of $\C[O^{X,\beta}_A]\in\hopf(\C[T/T_0],i^*\beta)$ is a compact quantum group monoidally equivalent to
\begin{equation}\label{eq:Uq(2)-type}
(T\times SU_q(2))/{(\operatorname{id}\times\chi)\Delta(T_0)}.
\end{equation}
\end{proposition}

We remark that the sign of $q$ is not uniquely determined by a monoidal equivalence between~\eqref{eq:Uq(2)-type} and the bosonization of $O^{X,\beta}_A$  in general, since $U_q(2)$ and $U_{-q}(2)$ are cocycle twists of each other.

\smallskip

Next, it is possible to obtain a classification of braided free orthogonal quantum groups up to isomorphism similar to the known classification of ordinary free orthogonal quantum groups, see again~\cite[Theorem~2.5.11]{NT13}.

\begin{proposition}
Consider a compact abelian group $T$ and a bicharacter $\beta\colon\hat T\times\hat T\to\T$. Assume $A\in\mathrm{GL}_m(\C)$ ($m\ge2$) is a matrix such that $A\bar A$ is unitary and $X=(x_{ij})_{i,j}\in\Mat_m(\C[T])$ is a unitary corepresentation matrix such that conditions~\eqref{eq:BFO-condition} and~\eqref{eq:BFO-condition2a} are satisfied for some $w\in\hat T$.  Assume $A'\in\mathrm{GL}_{m'}(\C)$ ($m'\ge2$) and $X'=(x'_{ij})_{i,j}\in\Mat_{m'}(\C[T])$ is another such pair,  with~\eqref{eq:BFO-condition} and~\eqref{eq:BFO-condition2a} satisfied for some $w'\in\hat T$. Then the braided Hopf $*$-algebras~$\C[O^{X,\beta}_A]$ and~$\C[O^{X',\beta}_{A'}]$ over $(\C[T],\beta)$ are isomorphic if and only if $m=m'$ and there exist a unitary matrix $v\in\operatorname{U}(m)$ and a character $\chi\in\hat T$ such that
\begin{equation}\label{eq:iso-conditions}
vXv^*=\chi X'\quad\text{and}\quad vADv^t=A',
\end{equation}
where $D=(\beta(x^*_{ij},\chi))_{i,j}=\overline{X(t_\chi)}$.
\end{proposition}

\bp
Denote by $U$ and $U'$ the fundamental unitaries of $\C[O^{X,\beta}_A]$ and $\C[O^{X',\beta}_{A'}]$, resp. If conditions~\eqref{eq:iso-conditions} are satisfied, then $\bar U_{\chi^{-1}X}=D^{-1}\bar U_XD$ by Lemma~\ref{lem:wZ}, hence $U= AD\bar U_{\chi^{-1}X}(AD)^{-1}$, and it is easy to check that we get an isomorphism $\C[O^{X,\beta}_A]\cong\C[O^{X',\beta}_{A'}]$ of braided Hopf $*$-algebras such that $U\mapsto v^* U' v$.

Conversely, assume we have an isomorphism $\rho\colon \C[O^{X,\beta}_A]\to\C[O^{X',\beta}_{A'}]$. We then get an isomorphism $\iota\#\rho\colon \C[T]\#\C[O^{X,\beta}_A]\to\C[T]\#\C[O^{X',\beta}_{A'}]$ of the bosonizations, which in turn defines a monoidal equivalence of the corresponding C$^*$-tensor categories of finite dimensional unitary comodules. For $\C[T]\#\C[O^{X,\beta}_A]$, as we discussed before Proposition~\ref{prop:monoidal-orthogonal}, this C$^*$-tensor category is $\Rep T\boxtimes\Rep SU_q(2)$ for suitable $q\in[-1,1]\setminus\{0\}$. The simple noninvertible objects of the smallest intrinsic dimension (equal to $|q+q^{-1}|$) in this category  are the tensor products of characters $\chi\in\hat T$ with the fundamental representation of $SU_q(2)$. At the level of $\C[T]\#\C[O^{X,\beta}_A]$, these objects are defined by the unitary corepresentation matrices $\chi X\# U$. For the same reason, the simple noninvertible objects of the smallest intrinsic dimension in the category of finite dimensional unitary comodules of $\C[T]\#\C[O^{X',\beta}_{A'}]$ are defined by the unitary corepresentation matrices $\chi X'\# U'$. It follows that there exist $\chi\in\hat T$ and a unitary $v\colon\C^m\to\C^{m'}$ such that
$$
v(\iota\#\rho)(X\# U) v^*=\chi X'\# U'.
$$
In particular, we must have $m=m'$. As $(\iota\#\rho)(X\# U)=X\#\rho(U)$, by applying $\iota\#\eps_{\C[O^{X',\beta}_{A'}]}$ we first conclude that $v Xv^*=\chi X'$ and then that $v\rho(U)v^*=U'$.

As $v \chi^{-1}Xv^*=X'$ and $v\rho(U)v^*=U'$, we have $\bar v\rho(\bar U_{\chi^{-1}X})\bar v^*=\bar U'_{X'}$. Recall also that $\bar U_{\chi^{-1}X}=D^{-1}\bar U_XD$ for $D=(\beta(x^*_{ij},\chi))_{i,j}$. By the relations in  $\C[O^{X,\beta}_A]$ we then get
$$
U'=v\rho(U)v^*=v\rho(AD\bar U_{\chi^{-1}X}(AD)^{-1}) v^*=v AD v^t \bar U'_{X'} \bar v (AD)^{-1}v^*.
$$
On the other hand, $U'=A'\bar U'_{X'}{A'}^{-1}$ by the relations in  $\C[O^{X',\beta}_{A'}]$. Since the matrix coefficients of $U'$ are linearly independent, it follows that
$$
v AD v^t=\lambda A'\quad\text{for some}\quad\lambda\in\C^\times.
$$
As both $A'\bar{A'}$ and $AD\overline{AD}$ are unitary (recall that $D=\overline{X(t_\chi)}$ is unitary and $AD$ coincides with~$\bar DA$ up to a phase factor), we must have $\lambda\in\T$. Hence, by multiplying $v$ by a phase factor we can achieve that $v AD v^t=A'$, while the equality $v Xv^*=\chi X'$ is still satisfied.
\ep

Note that, for fixed $(A,X)$ and $(A',X')$, there can only be finitely many $\chi$ such $vXv^*=\chi X'$ for some~$v$. Furthermore, if $\hat T$ is torsion-free (equivalently, $T$ is connected), then the only possible candidate for such $\chi$ is $w{w'}^{-1}$, since a nontrivial translation of a finite symmetric subset of $\hat T$ is never symmetric. Once $\chi$ is fixed, the question whether there is $v$ satisfying both conditions in~\eqref{eq:iso-conditions} can be solved by writing $(AD,\chi^{-1}X)$ and $(A',X')$ in a standard form in the following sense.

\begin{lemma}
Assume $A\in\mathrm{GL}_m(\C)$ ($m\ge2$) is a matrix such that $A\bar A$ is unitary and $X=(x_{ij})_{i,j}\in\Mat_m(\C[T])$ is a unitary corepresentation matrix such that $A(w_0 \bar{X})A^{-1}=X$ for some $w_0\in\hat T$. Then there is $v\in\operatorname{U}(m)$ such that $vAv^t$ and $vXv^*$ are block-diagonal matrices $\diag(A_1,\dots,A_k)$ and $\diag(X_1,\dots,X_k)$, where the blocks are
\begin{itemize}
  \item[--] either $2\times 2$ matrices
$$
A_i=\begin{pmatrix}
      0 & \theta_i\lambda_i^{-1} \\
      \lambda_i & 0
    \end{pmatrix},\qquad X_i=\begin{pmatrix}
                               \chi_i & 0 \\
                               0 & w_0\bar\chi_i
                             \end{pmatrix},
$$
with $0<\lambda_i\le1$, $\theta_i\in\T$ and $\chi_i\in\hat T$ such that if $\lambda_i=1$, then either $0<\arg\theta_i\le\pi$ or ($\theta_i=1$ and $\chi^2_i\ne w_0$);
\item[--] or $1\times1$ matrices $A_i=1$, $X_i=\chi_i$, with $\chi_i\in \hat T$ such that $\chi^2_i=w_0$.
\end{itemize}

\end{lemma}

\bp
The lemma can be viewed as a refinement of \cite[Proposition~1.5]{HN21}, but the proof is almost the same, so we will be brief.

First one observes that classification of the pairs $(A,X)$ up to the transformations $A\mapsto vAv^t$ and $X\mapsto vXv^*$ is the same as classification of the pairs $(AJ,X)$ up to unitary conjugacy, where $J\colon\C^m\to\C^m$ is the complex conjugation. Consider the polar decomposition $AJ=u|AJ|$, so~$|AJ|$ is positive and $u$ is anti-unitary. Then the joint spectrum, considered together with multiplicities, of the commuting operators $|AJ|$, $u^2$ and $X(t)$ ($t\in T$) is a complete invariant of the unitary conjugacy class, and a standard form of~$(A,X)$ as in the formulation of the lemma is obtained from an appropriate orthonormal basis diagonalizing these operators, as follows.

For $\lambda>0$, $\theta\in\T$ and $\chi\in\hat T$, consider the space
$$
H_{\lambda,\theta,\chi}=\{\xi\in\C^m: |AJ|\xi=\lambda\xi,\ u^2\xi=\theta\xi,\ X(t)\xi=\chi(t)\xi\ \text{for all}\ t\in T\}.
$$
As $u|AJ|=|AJ|^{-1}u$ and $w_0u X=u \bar w_0X=Xu$, we have $uH_{\lambda,\theta,\chi}=H_{\lambda^{-1},\bar\theta,w_0\bar\chi}$. In particular, the spaces $H_{\lambda,\theta,\chi}+H_{\lambda^{-1},\bar\theta,w_0\bar\chi}$ are invariant under $AJ$ and $X$. For every pair of triples $(\lambda,\theta,\chi)$ and $(\lambda^{-1},\bar\theta,w_0\bar\chi)$ such that $H_{\lambda,\theta,\chi}\ne0$, pick a representative and consider three cases.

1) Assume $(\lambda,\theta,\chi)\ne (\lambda^{-1},\bar\theta,w_0\bar\chi)$. Then, by changing the representative if necessary, we may also assume that either $\lambda<1$ or ($\lambda=1$ and $0\le\arg\theta\le\pi$). Choose an orthonormal basis~$(\xi_j)_j$ in~$H_{\lambda,\theta,\chi}$. Then $(u\xi_j)_j$ is an orthonormal basis in $H_{\lambda^{-1},\bar\theta,w_0\bar\chi}$, and the restrictions of~$AJ$ and~$X$ to every $2$-dimensional space with basis $\{\xi_j,u\xi_j\}$ have the required form.

2) Assume $\lambda=1$, $\theta=-1$, $\chi^2=w_0$. In this case $u\xi\perp\xi$ for every $\xi\in H_{1,-1,\chi}$. Hence we can find an orthonormal system $(\xi_j)_j$ in $H_{1,-1,\chi}$ such that $(\xi_j)_j\cup (u\xi_j)_j$ is an orthonormal basis in $H_{1,-1,\chi}$. The restrictions of $AJ$ and $X$ to every $2$-dimensional space with basis $\{\xi_j,u\xi_j\}$ have the required form.

3) Finally, assume $\lambda=1$, $\theta=1$, $\chi^2=w_0$. Then $\{\xi\in H_{1,1,\chi}\mid u\xi=\xi\}$ is a real form of $H_{1,1,\chi}$. By choosing an orthonormal basis in this Euclidean space we get a decomposition of $H_{1,1,\chi}$ into a direct sum of one-dimensional spaces that are invariant under $AJ$ and $X$.
\ep

From the proof one can see that the only source of nonuniqueness of the standard form of $(A,X)$, apart from the order of the blocks, is the choice of a representative from each pair of the triples $(1,\pm1,\chi)$ and $(1,\pm1,w_0\bar\chi)$ with $\chi^2\ne w_0$.

\begin{remark}\label{rem:TL-standard}
If $\hat T$ has zero $2$-torsion, then there is at most one $\chi$ such that $\chi^2=w_0$. Then, as in~\cite{HN21}, we can choose an orthonormal basis $(\xi_j)^l_{j=1}$ in $H_{1,1,\chi}$ such that $u\xi_j=\xi_{l-j+1}$. It follows that instead of making $vAv^t$ block-diagonal, we can make it to be of the form~\eqref{eq:TL-matrix}, with $0<a_i\le1$ for $i\le[\frac{m+1}{2}]$ and $0\le\arg a_i\le\pi$ for $i>[\frac{m+1}{2}]$ whenever $|a_i|=1$.\ee
\end{remark}

We can now parameterize the isomorphism classes of braided free orthogonal quantum groups $O^{X,\beta}_A$, at least when $\hat T$ is torsion-free. We already know that it suffices to consider only the transmutations of free orthogonal quantum groups, that is, we may assume that $A\bar A=\pm1$ and $w=1$ (the trivial character), and that then the only choice for $\chi$ in~\eqref{eq:iso-conditions} is $\chi=1$. Hence we get the following result.

\begin{corollary}
Consider a compact connected abelian group $T$ and a bicharacter $\beta\colon\hat T\times\hat T\to\T$. Then representatives $\C[O^{X,\beta}_A]\in\hopf(\C[T],\beta)$ of all isomorphism classes of braided free orthogonal quantum groups are obtained by considering the pairs $(A,X)$ such that
\begin{itemize}
  \item[--] the matrix $A$ has the form~\eqref{eq:TL-matrix}, with $0<a_i\le1$ for $i\le[\frac{m+1}{2}]$ and $a_ia_{m-i+1}=\tau$ for all $i$, where $\tau=\pm1$ (hence $a_{\frac{m+1}{2}}=1$ and $\tau=1$ if $m$ is odd);
  \item[--] the corepresentation matrix $X$ is diagonal, $X=\diag(\chi_1,\dots,\chi_m)$, with $\chi_i\chi_{m-i+1}=1$ for all $i$ (hence $\chi_{\frac{m+1}{2}}=1$ if $m$ is odd).
\end{itemize}
Two such pairs $(A,X)$ and $(A',X')$ define isomorphic braided Hopf $*$-algebras over $(\C[T],\beta)$ if and only if $(A',X')$ can be obtained from $(A,X)$ by
\begin{itemize}
  \item[--] permuting the pairs $(a_i,a_{m-i+1})$ and $(\chi_i,\chi_{m-i+1})$ ($i\le[\frac{m}{2}]$);
  \item[--] replacing $\chi_i$ by $\chi_{m-i+1}$ and $\chi_{m-i+1}$ by $\chi_i$ for some $i\le[\frac{m}{2}]$ such that $a_i=1$.
\end{itemize}
\end{corollary}

Now, observe that the quantum group~\eqref{eq:Uq(2)-type} is isomorphic to $U_q(2)$ when $\chi$ is nontrivial and $T/\ker\chi\cong\T$. All compact quantum groups monoidally equivalent to $U_q(2)$ are known by the work of Mrozinski~\cite{Mrozinski}. Our next goal is to describe the subclass of these quantum groups that can be obtained as the bosonizations of $\C[O^{X,\beta}_A]\in\hopf(\C[T/T_0],i^*\beta)$ for suitable $T_0$.

Assume that $B \in \mathrm{GL}_m(\C)$ is such that $B\bar{B}$ is unitary. Following the conventions of~\cite{HN22},  consider the quantum group $\tilde{O}_B^+$ of unitary transformations leaving the noncommutative polynomial $P=\sum_{i,j}B_{ji}X_iX_j$ invariant up to a phase factor: $\C[\tilde{O}_B^+]$ is the universal unital $*$-algebra generated by a unitary $d$ and elements $w_{ij}$, $1 \leq i,j \leq m$, such that
\begin{equation} \label{eq:Mrozinski}
 W = (w_{ij})_{i,j} \mbox{ is unitary and } W = B\bar{W} B^{-1}d,
\end{equation}
and the coproduct is given by $\Delta(w_{ij})=\sum_k w_{ik}\otimes w_{kj}$, $\Delta(d)=d\otimes d$. We remark that in~\cite{Mrozinski} the Hopf $*$-algebra $\C[\tilde{O}_B^+]$ is denoted by~$A_{\tilde{o}}(\bar B^t)$.

\begin{proposition}\label{prop:Mrozinski}
Assume that $B \in \mathrm{GL}_m(\C)$ ($m\ge2$) is such that $B\bar{B}$ is unitary. Then the following conditions are equivalent:
\begin{itemize}
\item[1)] $\C[\tilde{O}_B^+]$ is the bosonization of $\C[O^{X,\beta}_B]\in\hopf(\C[\T/\{\pm1\}],i^*\beta)$ for some unitary corepresentation matrix $X\in\Mat_m(\C[\T])$ and a bicharacter $\beta$ on $\Z=\hat\T$ as in Proposition~\ref{prop:G_A} with $X(-1)=1$ and $w=z^{-1}$;
\item[2)] $\C[\tilde{O}_B^+]$ is the bosonization of a braided Hopf $*$-algebra over $(\C[\T], \beta)$ for some bicharacter~$\beta$ on~$\Z=\hat\T$;
\item[3)] $m$ is even and the spectrum of $B\bar B$ consists of odd powers of a single number $\lambda\in \T$.
\end{itemize}
\end{proposition}
\begin{proof}
1) $\Rightarrow$ 2) is obvious.

\smallskip

2) $\Rightarrow$ 3): As $\tilde{O}_B^+$ has fusion rules of $U(2)$, the group-like elements of $\C[\tilde{O}_B^+]$ are powers of~$d$. It follows that if we can identify $\C[\tilde{O}_B^+]$ with $\C[\T] \# A$ for some $A\in\hopf(\C[\T],\beta)$, then $z\#1=d$  or $z\#1=d^{-1}$. By replacing $B$ by $\bar B^{-1}$ if necessary, we may assume that $z\#1=d$, since there is an isomorphism $\tilde{O}_B^+\cong \tilde{O}_{\bar B^{-1}}^+$ mapping $d$ into $d^{-1}$, see~\cite[Proposition~3.4]{Mrozinski}.

Consider the Hopf $*$-algebra map $\iota\#\eps_A \colon \C[\T] \# A\to\C[\T]$. Put $X=(\iota\#\eps_A)(W)$. Then from the defining relations for $\tilde{O}_B^+$ we get that $X$ is a unitary corepresentation matrix and $X = B\bar{X} B^{-1}z$. By Remark~\ref{rem:TL-standard} we may assume that
\begin{equation} \label{eq:B}
B = \begin{pmatrix}
0 & & b_m \\
 & \iddots & \\
b_1 & & 0
\end{pmatrix}, \quad b_i\bar{b}_{m-i+1} = \lambda_i \in \T,
\end{equation}
and $X = \mathrm{diag}(z^{k_1}, z^{k_2},...,z^{k_m})$ for some $k_i\in \Z$, $1 \leq i \leq m$. The identity $X = B\bar{X} B^{-1}z$ means then that $k_i + k_{m-i+1} = 1$ for all $i$. This implies that $m$ must be even.

Next, consider $U=(\eps_{\C[\T]}\#\iota)(W)\in\Mat_m(A)$. On the one hand, by~\eqref{eq:condition}, we have
$$
\delta_A(u_{ij})=u_{ij}\otimes z^{k_j-k_i}.
$$
As $W=X\# U=(z^{k_i}\# u_{ij})_{i,j}$ and $z\#1=d$, by the definition of the smash product $\C[\T] \# A$ it follows that
$$
dWd^*=(\beta(k_i-k_j,1)w_{ij})_{i,j}.
$$

On the other hand, the relation $W = B\bar{W} B^{-1}d$ implies $\bar W=d^*\bar B W\bar B^{-1}$, hence
$$
dWd^*=B\bar B W (B\bar B)^{-1}=(\lambda_i\bar\lambda_j w_{ij})_{i,j}.
$$
We therefore have that $\beta(k_i-k_j,1) = \lambda_i\bar{\lambda}_j$ for all $i,j$, whence $\lambda_i = \lambda\zeta^{k_i}$ for some $\lambda\in \T$, where $\zeta=\beta(1,1)$. As
$\bar \lambda_i=\lambda_{m-i+1}$ and $k_i+k_{m-i+1}=1$, we must have $\lambda^{-2}=\zeta$ and thus $\lambda_i = \lambda^{-2k_i + 1}$.

\smallskip

3) $\Rightarrow$ 1): We may assume that $B$ is of the form~\eqref{eq:B}. By the assumption on the spectrum there exist $l_i\in\Z$ for $i=1,\dots,m/2$ such that $\lambda_i=\lambda^{-2l_i-1}$. Let $l_i=-1-l_{m-i+1}$ for $i=m/2+1,\dots,m$. Then  $\lambda_i=\lambda^{-2l_i-1}$ for all $i$.

Put $X=\mathrm{diag}(z^{l_1}, z^{l_2},...,z^{l_m})$. Define a bicharacter $\beta$ on $\hat\T$ by $\beta(k,l)=\lambda^{-2kl}$. We then have $X = B\bar{X} B^{-1}z^{-1}$ and $B\bar B=\lambda^{-1} X(\lambda^{-2})$. Consider the double cover $p\colon\T\to\T$, $t\mapsto t^2$. Denote by $\beta^{1/4}$ the bicharacter on~$\hat\T$ defined in the same way as~$\beta$, but with $\lambda^2$ replaced by one of its fourth roots. Therefore if $i\colon\C[\T]\to\C[\T]$ is defined by $i(z)=z^2$, we get $i^*\beta^{1/4}=\beta$.
By Remark~\ref{rem:BFO} the braided quantum group  $O_B^{Xp, \beta^{1/4}}$ is well-defined. As $(Xp)(-1)=1$, we can view $\C[O_B^{Xp, \beta^{1/4}}]$ as an object $A$ in $\hopf(\C[\T/\{\pm1\}],i^*\beta^{1/4})=\hopf(\C[\T],\beta)$.

The defining relations in $A$ say that $B$ is an intertwiner of the $A$-comodules defined by $(z^{-1}\bar{X},\bar U_X)$ and $(X, U)$. In other words, for $W'=X\# U=(z^{l_i}\# u_{ij})_{i,j}$ we have
$$
B(z^{-1}\#1)\overline{W'}=W'B.
$$
The relation  $W = B\bar{W} B^{-1}d$ in $\C[\tilde O^+_B]$ can be written as
$$
d^*W = Bd^*\overline{d^*W}B^{-1}.
$$
It follows that we have a well-defined Hopf $*$-algebra map $\C[\tilde O^+_B]\to \C[\T]\# A$ such that $d\mapsto z\#1$ and $W=(w_{ij})_{i,j}\mapsto zX\#U=(z^{l_i+1}\# u_{ij})_{i,j}$. Using the relations in both algebras it is also easy to construct the inverse map.
\end{proof}

\begin{remark}
The proof of the proposition implies that the procedure described in the proof of the implication 3) $\Rightarrow$ 1) is the only way of getting decompositions $\C[\tilde O^+_B]=\C[\T]\# A$ such that $A\in\hopf(\C[\T],\beta)$ and $d=z\#1$.
\end{remark}

\begin{remark}
As in the proof of the proposition, it is not difficult to see that given $B \in \mathrm{GL}_m(\C)$ such that $B\bar{B}$ is unitary, a unitary corepresentation matrix $X \in \mathrm{Mat}_m(\C[\T])$ such that $X=B\bar X B^{-1}z$ exists if and only if $m$ is even. Therefore, for even $m$, we always get a Hopf $*$-algebra map $p: \C[\tilde{O}_B^+] \to \C[\T]$ such that $p(d)=z$ with a right inverse $z\mapsto d$. By Radford's theorem~\cite{Radford} we then get a Hopf $*$-algebra object in the braided category $\mathcal{YD}(\T)$ of $\T$-Yetter–Drinfeld modules. From this perspective the above proposition characterizes when this object lies in the subcategory $(\mathcal{M}^{\C[\T]}, \beta) \subset \mathcal{YD}(\T)$ for some $\beta$.\ee
\end{remark}

Finally, let us compare the transmutations of $\C[O_F^+]$ to the braided quantum groups constructed in \cite{MR-ort}. Fix numbers $d_1, d_2,...,d_m, d \in \Z$ and consider the representation
\[ X\colon \T \to \Mat_m(\C), \quad X(t) = \begin{pmatrix}
t^{d_1} & & 0 \\
& \ddots & \\
0 & & t^{d_m}
\end{pmatrix}.\]
Take $\Omega = (\omega_{ij})_{i,j} \in \mathrm{GL}_m(\C)$ and assume
\begin{equation*}
\label{eq:Omega}
\omega_{ij} \neq 0 \, \Rightarrow \,  d_i + d_j = d,\quad \text{or equivalently},\quad
\Omega^t (z^d \bar{X} )(\Omega^t)^{-1} = X,
\end{equation*}
where $z \in \C[\T]$ is the generator. Assume further that there  is $\zeta \in \T$ such that
\[\bar{\Omega}\Omega = c \,X(\zeta^d)\quad\text{for some}\quad c\in\T.\]
Define a bicharacter on $\hat{\T} = \Z$ by $\beta(m,n) = \zeta^{mn}$.

In \cite{MR-ort}, Meyer and Roy construct an object $A_o(\Omega, X) \in \hopf(\C[\T], \beta)$ from this data.
By \cite[Theorem 2.6]{MR-ort} and \cite[Remark 2.20(2)]{BJR}, in our terminology, $A_o(\Omega, X)$ is defined as a universal braided compact matrix quantum group over $(\T, X, \beta)$ with fundamental unitary $U = (u_{ij})_{i,j}$ subject to the relations
\[ U = A\bar{U}_XA^{-1}, \]
where $A = (\omega_{ji}\zeta^{dd_j})_{i,j}=\Omega^tX(\zeta^d)=\zeta^{d^2}X(\zeta^{-d})\Omega^t$. Note that we have $\bar{U}_X = (\zeta^{d_i(d_j-d_i)}u_{ij}^*)_{i,j}$ and $A\bar{A} = c\,\zeta^{d^2}X(\zeta^{-d})$.

Assume that $d = 2n$. Put $w(t)=t^n$. Then we see from the description above and Remark~\ref{rem:BFO} that $A_o(\Omega,X) \cong \C[O_A^{X,\beta}]$ in $\hopf(\C[\T],\beta)$.

Next, assume that $d$ is odd. Then, as in the proof of Proposition~\ref{prop:Mrozinski}, $m$ is even. Similarly to that proof, consider the double cover $p\colon\T\to\T$, $t\mapsto t^2$, to write the character~$z^d$ as a square. Write $\beta^{1/4}$ for the bicharacter on~$\hat\T$ defined in the same way as~$\beta$, but with $\zeta$ replaced by a fourth root $\zeta^{1/4}$. Define a character $w$ on the double cover by $w(t)=t^d$. Then
$$
A_o(\Omega,X) \cong \C[O_A^{Xp,\beta^{1/4}}],
$$
where we now view the latter braided Hopf $*$-algebra as an object in $\hopf(\C[\T/\{\pm1\}],i^*\beta^{1/4})=\hopf(\C[\T],\beta)$.

\subsection{Braided free unitary quantum groups}
\label{ex:free-unitary} Let $m \geq 2$ be a natural number. We recall the definition of the \textit{free unitary quantum group} $U_F^+$. Let $F = (f_{ij})_{i,j} \in \mathrm{GL}_m(\C)$ be such that $\Tr(F^*F) = \Tr((F^*F)^{-1})$. Then $\C[U_F^+]$ denotes the universal $*$-algebra with generators $u_{ij}$, $1 \leq i, j \leq m$, and relations determined by
\begin{equation*}
\label{eq:free-unitary}
U = (u_{ij})_{i,j} \mbox{ and } F\bar{U}F^{-1} \mbox{ are unitaries in } \Mat_m(\C[U_F^+]).
\end{equation*}
The Hopf $*$-algebra structure on $\C[U_F^+]$ is defined so that $U_F^+$ is a compact matrix quantum group as in Definition~\ref{def:matrix}.

Similarly to the previous example, we fix a compact abelian group $T$ together with a unitary corepresentation matrix $Z\in\Mat_m(\C[T])$ such that $F\bar{Z}F^{-1}$ is unitary, equivalently, $\bar Z$ commutes with $|F|$. Then, by the universality of $\C[U_F^+]$, there is a Hopf $*$-algebra map $\pi\colon \C[U_F^+] \to \C[T]$ mapping $U$ to $Z$.
Let $\beta\colon \hat{T} \times \hat{T} \to \T$ be a bicharacter, and consider the transmutation $\C[U^+_F]_\beta$. Then, similarly to Proposition~\ref{prop:O_F^+-universal}, we get the following result.

\begin{proposition}
\label{prop:free-unitary-universal}
The braided Hopf $*$-algebra $\C[U_F^+]_\beta$ is a braided compact matrix quantum group over $(T,Z,\beta)$ with fundamental unitary $U=(u_{ij})_{i,j}$.
As a $*$-algebra, it is a universal unital $*$-algebra with generators~$u_{ij}$ satisfying the relations
\begin{equation*}
\label{eq:braided-free-unitary-1}
U\quad \mbox{and}\quad F\bar{U}_Z F^{-1}\quad\text{are unitaries},
\end{equation*}
where $\bar{U}_Z= (\bar{u}_{ij}^Z)_{i,j}$ and $\bar{u}_{ij}^Z = \sum_{s,l,t} \beta(z_{tj}^*z_{sl}, z_{il}^*)u_{st}^*$.
\end{proposition}

By Theorem~\ref{thm:main}, the bosonization of $\C[U_F^+]_\beta$ is a compact quantum group that is a cocycle twist of $T\times U^+_F$.

\smallskip

When $T=\T$ and $Z$ is such that both $Z$ and $F\bar ZF^{-1}$ are diagonal matrices, we recover the braided free unitary quantum groups defined in~\cite{BJR}. We remark that we can of course always assume that one of the matrices $Z$ or $F\bar ZF^{-1}$  is diagonal by choosing an appropriate orthonormal basis, but it is usually impossible to make both of them diagonal simultaneously.

\subsection{Anyonic quantum permutation groups} Let $N \geq 2$ be a natural number. The quantum symmetric group $S_N^+$ is the universal compact matrix quantum group with fundamental unitary representation $U = (u_{ij})_{i,j= 0}^{N-1}$ subject to the relations
\[\sum_i u_{ij} = 1 = \sum_j u_{ij} \quad \mbox{ and } \quad u_{ij}^* = u_{ij}^2 = u_{ij} .\]
We can view the cyclic group $\Z/N\Z$ as a subgroup of $S_N\subset S_N^+$, so we get a Hopf $*$-algebra map
\[\pi\colon \C[S_N^+] \to \C[\Z/N\Z], \quad \pi(u_{ij}) = \delta_{j-i}, \]
where $\delta_{k} \in \C[\Z/N\Z]$ are the usual delta-functions.
%\[ \delta_{t}(s) = \begin{cases}
%1 & \mbox{ if } s = t \\
%0 & \mbox{ otherwise}.
%\end{cases}\]

To describe the transmutation we want to express the relations in $\C[S_N^+]$ in terms of homogeneous elements with respect to the bi-grading by $\Z/N\Z$. Fix a primitive $N$-th root of unity~$\omega$. By Lemma 4.9 in \cite{ABRR} such generators can be obtained by considering the elements $a_{ij}$ defined~by
\[ (a_{ij})_{i,j} = \Omega^{-1}U\Omega, \quad \Omega = \Big( \dfrac{1}{N}\omega^{-ij} \Big)_{i,j = 0}^{N-1}.\]
The elements $a_{ij}$ are then bi-graded by
\[(\pi \otimes \iota)\Delta(a_{ij}) = z^i\otimes a_{ij}, \quad (\iota\otimes\pi)\Delta(u_{ij}) = a_{ij} \otimes z^j, \]
where $z \in \C[\Z/N\Z]$ is the function $z(k) = \omega^k$. In terms of the new generators the relations in~$\C[S_N^+]$ become
\[a_{0i} = a_{i0} = \delta_{i,0}, \quad a_{ij}^* = a_{-i,-j}, \]
\[a_{k,i+j} = \sum_l a_{k-l,i}a_{lj}, \quad a_{i+j,k} = \sum_l a_{jl}a_{i,k-l}. \]

Define a bicharacter by $\beta(z^i,z^j) = \omega^{-ij}$. It is then readily verified, by using formulas~\eqref{eq:trans-prod} and~\eqref{eq:trans-star}, that the transmutation $\C[S_N^+]_\beta$ is described by the relations
\[a_{0i} = a_{i0} = \delta_{i,0}, \quad a_{ij}^* =  \omega^{i(j-i)}a_{-i,-j}, \]
\[a_{k,i+j} = \sum_l \omega^{-l(i-k + l)}a_{k-l,i}a_{lj}, \quad a_{i+j,k} = \sum_l \omega^{-i(l-j)} a_{jl}a_{i,k-l}.\]
These are exactly the relations in \cite[Definition 2.7]{ABRR}.
Finally, by Theorem~\ref{thm:main}, the bosonization of $\C[S_N^+]_\beta$ is a cocycle twist of the quantum group $(\Z/N\Z)\times S_N^+$.

\printbibliography

\end{document}